\documentclass[a4paper,12pt]
{amsart}
\usepackage[utf8x]{inputenc}
\usepackage{amssymb,latexsym}
\usepackage{mathrsfs}
\usepackage{euscript}
\usepackage{hyperref}
\usepackage{amsmath}

\newtheorem{theorem}{Theorem}[section]
\newtheorem{definition}[theorem]{Definition}
\newtheorem{lemma}[theorem]{Lemma}

\newtheorem{remark}[theorem]{Remark}
\newtheorem{proposition}[theorem]{Proposition}
\newtheorem{notation}[theorem]{Notation}

\def\cal{\mathcal}

\def\Bbb{\mathbb}

\def\C{{\Bbb C}}

\def\R{{\mathbb R}}

\def\<{\left<}
\def\>{\right>}

\def\S{{\cal S}}

\def\J{{\cal J}}

\def\BMO{{\rm BMO}}

\def\CMO{{\rm CMO}}
\def\Im{{\rm Im}\,}
\def\Re{{\rm Re}\,}

\def\({( \hspace{-0.335em}(}
\def\){) \hspace{-0.335em})}

\def\fu2{\frac{n}{2}}

\def\l({\left(}
\def\r){\right)}
\def\be{\begin{enumerate}}
\def\ee{\end{enumerate}}

\def\sign{{\rm sign}}

\def \diam {{\rm diam}}

\def \rdist {{\rm \, rdist}}
\def \ec {{\rm ec}}

\title[\tiny Endpoint compactness of singular integrals]
{Endpoint compactness of singular integrals and perturbations of the Cauchy integral}
\author{Karl-Mikael Perfekt} \address{Department of Mathematical Sciences, Norwegian University of Science and Technology (NTNU), NO-7491 Trondheim, Norway} \email{karl-mikael.perfekt@math.ntnu.no}
\author{Sandra Pott}
\address{Centre for Mathematical Sciences, University of Lund, PO Box 118, 22100 Lund, Sweden}
\email{sandra@maths.lth.se}
\author{Paco Villarroya}
\address{Centre for Mathematical Sciences, University of Lund, PO Box 118, 22100 Lund, Sweden}
\email{paco.villarroya@maths.lth.se}

\thanks{The second author has been partially supported by the Crafoord Foundation. 
The last author has been partially supported by the Spanish Ministry of Economy and Competitiveness (project grants MTM2011-23164 and MTM2014-53009-P)}


\begin{document}

\maketitle

\begin{abstract}
We prove sufficient and necessary conditions for compactness of Calder\'on-Zygmund operators
on the endpoint from $L^{\infty }(\mathbb R)$ into $\CMO(\mathbb R)$. 
We use this result to prove
compactness on $L^{p}(\mathbb R)$ with $1<p<\infty $ of certain perturbations of the Cauchy integral on curves with normal derivatives satisfying a $\CMO$-condition.

\end{abstract}

\section{Introduction}
In \cite{V}, 
we started a general theory to characterize
compactness of singular integral operators.
More precisely, we showed that a Calder\'on-Zygmund operator $T$
is compact on $L^p(\R)$ with $1<p<\infty $ if and only three conditions hold: the operator kernel satisfies the definition of 
%what we call 
a compact 
Calder\'on-Zygmund kernel, a strengthening of the smoothness condition of a standard 
Calder\'on-Zygmund kernel;
$T$ satisfies a new property of
\textit{weak compactness}, analogue to the classical \textit{weak boundedness};  
and the functions $T(1)$ and $T^*(1)$ belong to the space
$\CMO(\R)$, the appropriate substitute of $\BMO(\R)$.  

Now, the purpose of the current paper is to continue this study in two different but related ways. First, 
we extend the results appearing in \cite{V} to one of the endpoint cases, namely, from 
$L^{\infty}(\mathbb R)$  
%\cap \CMO 
into $\CMO(\mathbb R)$ (Theorem \ref{Mainresult2}). 
For this purpose, we follow a new approach, based on the study of boundedness of a modified Martingale Transform, Proposition \ref{modifiedmartingaletransform}, 
which substitutes the classical square function and, to the authors knowledge, has not been studied before.
Second, we use the latter result to provide 
an application of the general theory by showing how the methods devised in \cite{V} allow to prove compactness on $L^{p}(\mathbb R)$ of 
certain perturbation of the Cauchy integral operator defined over Lipschitz curves with $\CMO$-smooth normal derivatives, Proposition \ref{applprop}. 

The paper is structured as follows: in Section 2, we give the necessary definitions, we state the main results of \cite{V} that there will be needed and we also state our main result in this paper, Theorem \ref{Mainresult}; in Section 3 we characterize compactness of Calder\'on-Zygmund operators on the endpoint case $L^{\infty} 
%\cap \CMO 
\to \CMO$. Finally, in Section 4 we provide an application of the theory by proving compactness of the before described perturbation of the Cauchy integral. 

\section{Definitions and statement of the main result}
\subsection{Definitions and notation}
\begin{definition}\label{admissible}
We say that three bounded functions $L,S, D: [0,\infty )\rightarrow [0,\infty )$
constitute a set of admissible functions if 
the following limits hold
\begin{equation}\label{limits}
\lim_{x\rightarrow \infty }L(x)=\lim_{x\rightarrow 0}S(x)=\lim_{x \rightarrow \infty }D(x)=0.
\end{equation}

\end{definition}

\begin{remark}\label{constants}
Since any fixed dilation of an admissible function $L_{\lambda}(x)=L(\lambda^{-1}x)$ is again admissible, we will often omit all universal constants appearing in the argument of these functions. 
\end{remark}

\begin{definition}\label{prodCZ}
Let $\Delta $ be the diagonal of $\mathbb R^{2}$.
Let $L,S,D$ be admissible functions. 

A function $K:(\R^{2} \setminus \Delta )\to \mathbb C$ is called a
compact Calder\'on-Zygmund kernel if it is bounded on compact sets of $\R^{2} \setminus \Delta $ and 
for some $0<\delta \leq 1$ and $C>0$, we have
\begin{multline*}
|K(t, x)-K(t',x')|
\le 
C\frac{(|t-t'|+|x-x'|)^\delta}{|t-x|^{1+\delta}}
L(|t-x|)S(|t-x|)
D(|t+x|),
\end{multline*}
whenever $2(|t-t'|+|x-x'|)<|t-x|$. 
\end{definition}

As shown in \cite{V}, it can be assumed without loss of generality 
that  in Definition \ref{prodCZ} 
the functions 
$L$ and $D$ are monotone non-increasing while the function $S$ is monotone non-decreasing. 

We also remark that there is an equivalent definition of compact Calder\'on-Zygmund kernels  
which is more convenient to use in applications of the theory (see Section 4). 
In \cite{V}, we show that 
Definition \ref{prodCZ} 
 is equivalent to the existence of a bounded function 
$B: \mathbb R^{2}\rightarrow [0,\infty )$ such that 
$$
\lim_{|t-x|\rightarrow \infty }B(t,x)=\lim_{|t-x|\rightarrow 0}B(t,x)=\lim_{|t+x| \rightarrow \infty }B(t,x)=0
$$
and for some $0<\delta \leq 1$ and $C>0$, 
$$
|K(t,x)-K(t',x')|
\le 
C\frac{(|t-t'|+|x-x'|)^\delta}{|t-x|^{1+\delta}}B(t,x),
$$
whenever $2(|t-t'|+|x-x'|)<|t-x|$.

\begin{definition}\label{SchwartzN} 
For every $N\in \mathbb N$, $N\geq 1$, we define $\S_{N}(\R)$ to be the set of 
all functions 
$f\in {\mathcal C}^{N}(\mathbb R)$
such that 
$$
\| f\|_{m,n}=\sup_{x\in \mathbb R}|x|^{m}|f^{(n)}(x)|<\infty
$$ 
for all $m,n\in \mathbb N$ with $m,n\leq N$. Clearly, $\S_{N}(\R)$ 
equipped with the family of seminorms $\| \cdot \|_{m,n}$
is a Fr\'echet space. Then,   
we can also define its dual space $\S_{N}'(\R)$ equipped with the dual topology which turns out to be
a subspace of the space of tempered distributions. 
\end{definition}

\begin{definition}\label{intrep}
Let $T:\S_{N}(\R)\to \S_{N}'(\R)$ be a linear operator which is continuous with respect the topology of 
$\S_{N}(\R)$ for a fixed $N\geq 1$.

We say that 
$T$ is associated with a compact Calder\'on-Zygmund kernel $K$ if 
the action of $T(f)$ as a distribution satisfies the following 
integral representation
$$
\langle T(f), g\rangle =\int_{\R}\int_{\R} f(t)g(x) K(t,x)\, dt \, dx
$$
for all functions $f,g\in {\cal S_{N}(\R)}$
with disjoint compact supports.
\end{definition}

\begin{definition}
For $0<p\leq \infty $ and $N\in \mathbb N$, we say that a function $\phi \in {\mathcal S}_{N}(\mathbb R)$ is 
an $L^p(\mathbb R)$-normalized bump function adapted to $I$ with constant $C>0$ and order $N$,
if it satisfies 
$$|\phi^{(n)}(x)|\le C\frac{1}{|I|^{\frac{1}{p}+n}}\Big(1+\frac{|x-c(I)|}{|I|}\Big)^{-N}, \ \ \ 0\leq  n \le N$$
for every interval $I\subset \mathbb R$, where we denote its centre by $c(I)$ and its length by $|I|$.

\end{definition}

The order of the bump functions
will always be denoted by $N$, even though its value might change from line to line.
We will often use the greek letters $\phi $, $\varphi $ for general bump functions while
we reserve the use of $\psi $ to denote bump functions with mean zero.
If not otherwise stated, we will usually assume that bump functions are $L^2(\mathbb R)$-normalized.

A result we will use in forthcoming sections is the following property of bump functions whose 
proof can be found in \cite{TLec}:

\begin{lemma}\label{decayofidentity} Let $I$, $J$ be intervals and let $\phi_I $, $\varphi_J$ be bump functions $L^2$-adapted to
$I$ and $J$ respectively with order $N$ and constant $C>0$. Then,
$$
|\langle \phi_I,\varphi_J\rangle|\leq C \left(\frac{\min(|I|,|J|)}{\max (|I|,|J|)}\right)^{1/2}
\left( \frac{{\rm diam}(I\cup J)}{\max(|I|,|J|)} \right)^{-N}.
$$
Moreover, if $|J|\leq |I|$ and $\psi_J$ has mean zero then 
$$
|\langle \phi_I,\psi_J\rangle|\leq C \left(\frac{|J|}{|I|}\right)^{3/2}
\left( \frac{{\rm diam}(I\cup J)}{|I|} \right)^{-(N-1)}.
$$
\end{lemma}

\begin{notation}\label{ecandrdist}
We now introduce some notation which will be frequently used throughout the paper.  
We denote by $\mathbb B=[-1/2,1/2]$ and $\mathbb B_{\lambda }=\lambda \mathbb B=[-\lambda/2,\lambda /2]$.

Given two intervals $I,J\subset \mathbb R$, 
we define $\langle I,J\rangle$ as the smallest interval containing $I\cup J$ and  
we denote its measure by $\diam(I\cup J)$. Notice that 
\begin{eqnarray*}
\diam(I\cup J) & \approx & |I|/2+|c(I)-c(J)|+|J|/2.
\end{eqnarray*}

We also define
the relative distance between $I$ and $J$ by
$$
\rdist(I,J)=\frac{\diam(I\cup J)}{\max(|I|,|J|)},
$$
which is comparable to $\max(1,n)$ where $n$ is  
the smallest number of times the larger interval needs to be shifted a distance equal to its side length 
so that it contains the smaller one. 
Notice that 
\begin{eqnarray*}
\rdist(I,J) & \approx & 1+\frac{|c(I)-c(J)|}{\max(|I|,|J|)}.
\end{eqnarray*}

Finally, we  define the eccentricity of $I$ and $J$ to be
$$
\ec(I,J)=\frac{\min(|I|,|J|)}{\max (|I|,|J|)}.
$$

\end{notation}

\begin{definition}\label{WB}
A linear operator $T : \S_{N}(\R) \to \S_{N}'(\R)$ with $N\geq 1$ satisfies the weak compactness condition, if 
there exist admissible functions $L, S, D$ 
such that: for every $\epsilon>0$ there 
exists $M\in \mathbb N$ so that for any interval $I$ and every pair $\phi_I, \varphi_I$ of
$L^2$-normalized bump functions adapted to $I$ with constant $C>0$ and order $N$, we have
\begin{equation}\label{weakcompactnessformula}
|\langle T(\phi_I),\varphi_I)\rangle |\lesssim C(L(2^{-M}|I|)S(2^{M}|I|)
D(M^{-1}\rdist(I,\mathbb B_{2^{M}}))+\epsilon ),
\end{equation}
where the implicit constant only depends on the operator $T$. 
\end{definition}

\begin{remark}
We note that in the main results of the paper, namely Theorem \ref{Mainresult2} or Theorem \ref{BMObounds}, when we say that $T$ satisfies the weak compactness condition, we mean that there is an integer $N\geq 1$ sufficiently large depending on the operator or its kernel so that the operator can be defined $T : \S_{N}(\R) \to \S_{N}'(\R)$, it is continuous with respect the topology in 
$\S_{N}(\R) $ and it satisfies Definition \ref{WB} for that value of $N$. 
\end{remark}

In \cite{V} we discuss other equivalent formulations of this Definition.

From now on, we will denote
$$
F_{K}(I)=L_{K}(|I|)S_{K}(|I|)D_{K}(\rdist(I,\mathbb B))
$$
and 
$$
F_{W}(I;M)
=L_{W}(2^{-M}|I|)S_{W}(2^{M}|I|)D_{W}(M^{-1}\rdist(I,\mathbb B_{2^{M}})), 
$$
where $L_{K}$, $S_{K}$ and $D_{K}$ are the functions appearing in the definition of a compact 
Calder\'on-Zygmund kernel,
while $L_{W}$, $S_{W}$, $D_{W}$ and the constant $M$ are as in the definition of the weak compactness condition. Note 
that the value $M = M_{T,\epsilon}$ depends not only on $T$ but also on $\epsilon$. 

We will also denote $F(I;M)=F_{K}(I)+F_{W}(I;M)$,
$$
F_{K}(I_{1},\cdots, I_{n})=\big( \sum_{i=1}^{n}L_{K}(|I_{i}|)\big) \big( \sum_{i=1}^{n}S_{K}(|I_{i}|)\big)
\big(\sum_{i=1}^{n}D_{K}(\rdist(I_{i},\mathbb B))\big)
$$
\begin{eqnarray*}
\hspace{-1.5cm}
F_{W}(I_{1},\cdots, I_{n};M)&=&\big(\sum_{i=1}^{n}L_{W}(2^{-M}|I_{i}|)\big) \big(\sum_{i=1}^{n}S_{W}(2^{M}|I_{i}|)\big)\\
&&\big(\sum_{i=1}^{n}D_{W}(M^{-1}\rdist(I_{i},\mathbb B_{2^{M}}))\big)
\end{eqnarray*}
and 
$
F(I_{1},\cdots, I_{n};M)=F_{K}(I_{1},\cdots, I_{n})+F_{W}(I_{1},\cdots, I_{n};M)
$.

\subsection{Characterization of compactness. The Lagom Projection operator}
In order to prove our results about compact singular integral operators, we will use the following 
characterization of compact operators in a Banach space with a Schauder basis (see \cite{Fab}).

\begin{theorem}\label{charofcompact}
Suppose that $\{e_{n}\}_{n\in \mathbb N}$ is a Schauder basis of a Banach space $E$. For each positive integer $k$, let $P_{k}$ be the canonical projection,
$$
P_{k}(\sum_{n\in \mathbb N}\alpha_{n}e_{n})=\sum_{n\leq k}\alpha_{n}e_{n}.
$$
Then, a bounded linear operator $T:E \to E$ is compact if and only if $P_{k}\circ T$ converges to $T$ in operator norm.
\end{theorem}

\begin{definition} \label{Imdef}
For every $M\in \mathbb N$, let ${\cal I}_{M}$ be the family of intervals such that 
$2^{-M}\leq |I|\leq 2^{M}$ and 
$\rdist(I,\mathbb B_{2^{M}})\leq M$. Let ${\mathcal D}$ be the family of dyadic intervals of the real line and ${\cal D}_{M}$ be the intersection of ${\cal I}_{M}$ with ${\mathcal D}$. We call the intervals in ${\cal I}_{M}$ and ${\cal D}_{M}$ as lagom intervals and dyadic lagom intervals respectively.
\end{definition}

Notice that $I\in {\cal D}_{M}$ implies that   $2^{-M}(2^{M}+|c(I)|)\leq M$ and then
$|c(I)|\leq (M-1)2^{M}$.  
Therefore, $I\subset \mathbb B_{M2^{M}}$ with
$2^{-M}\leq |I|$.  

On the other hand, $I\notin {\cal D}_{M}$ implies either $|I|>2^{M}$ or 
$|I|<2^{-M}$ or $2^{-M}\leq |I|\leq 2^{M}$ with $|c(I)|>(M-1)2^{M}$.

Let $E$ be one of the following Banach spaces: the Lebesgue space $L^{p}(\mathbb R)$, $1<p<\infty $, the Hardy space $H^{1}(\mathbb R)$, or the space $\CMO(\R)$, to be introduced later as the closure in $\BMO(\mathbb R)$ of continuous functions vanishing at infinity. In each case, $E$ is equipped with smooth wavelet bases which are also Schauder bases (see \cite{HerWeiss} and Lemma \ref{lem:cmochar}). Moreover, in all cases, we have at our disposal smooth and compactly supported wavelet bases. 
 
\begin{definition}\label{lagom}
Let $E$ be one of the previously mentioned Banach spaces.  
Let $(\psi_{I})_{I\in {\mathcal D}}$ be a wavelet basis of $E$. 
Then, for every $M\in \mathbb N$,
we define the lagom projection operator $P_{M}$ by 
$$
P_{M}(f)=\sum_{I\in {\cal D}_{M}}\langle f,\psi_{I}\rangle \psi_{I},
$$
where $\langle f,\psi_{I}\rangle =\int_{\mathbb R}f(x)\overline{\psi(x)}dx$.

We also define the orthogonal lagom projection operator as $P_{M}^{\perp }(f)=f-P_{M}(f)$. 
\end{definition}

\begin{remark}
Without explicit mention, we will let the wavelet basis defining $P_M$ vary from proof to proof to suit our technical needs.

We also note the use of the same notation for the action of $T(f)$ as a distribution
and the inner product.
We hope that this will not cause confusion. 
\end{remark}

It is easy to see that both $P_M$ and $P_{M}^{\perp }$ are self-adjoint operators.

We note the difference with the usual projection operator, ${\mathcal P}_{Q}$ for every interval
$Q \subset \mathbb R$, defined by 
\begin{equation}\label{regularproj}
{\mathcal P}_{Q}(f)=
\sum_{\displaystyle{\tiny \begin{array}{c}I\in {\cal D}\\ I\subset Q \end{array}}}\langle f,\psi_{I}\rangle \psi_{I}, 
\end{equation}
which we will also use in forthcoming sections.

Let $S$ denote the square function operator associated with a wavelet basis $(\psi_{I})_{I\in \mathcal D}$
$$
S(f)(x)=\Big( \sum_{I\in {\cal D}}\frac{|\langle f, \psi_{I}\rangle |^{2}}{|I|}\chi_{I}(x)\Big)^{1/2}.
$$
Since we trivially have the pointwise estimates $S(P_{M}(f))(x)\leq S(f)(x)$ and $S(P_{M}^{\perp}(f))(x)\leq S(f)(x)$, by Littlewood-Paley theory, we deduce that 
the lagom projection operator and its orthogonal projection are both continuous on $L^{p}(\mathbb R)$ for all $1<p<\infty $. Moreover, 
the estimate 
$$
\| P_{M}(f)\|_{\BMO(\mathbb R)}
=\sup_{Q \subset \mathbb R} \Big(\frac{1}{|Q|}
\sum_{\tiny \begin{array}{c}I\in {\cal D}_{M}\\ I\subset Q \end{array}}|\langle f,\psi_{I}\rangle |^{2}\Big)^{1/2}
\leq \| f\|_{\BMO(\mathbb R)}
$$
shows that $P_{M}$ is bounded on $\BMO(\mathbb R)$  and, by duality, on $H^{1}(\mathbb R)$ 
with $\|P_{M}\|_{\BMO(\mathbb R)\rightarrow \BMO(\mathbb R)}\leq 1$ and
$\|P_{M}\|_{H^{1}(\mathbb R)\rightarrow H^{1}(\mathbb R)}\leq 1$. For similar reasons, we have 
$\|P_{M}^{\perp}\|_{\BMO(\mathbb R)\rightarrow \BMO(\mathbb R)}\leq 1$
and
$\|P_{M}^{\perp}\|_{H^{1}(\mathbb R)\rightarrow H^{1}(\mathbb R)}\leq 1$

We remark that in $E$, the equality
\begin{equation}\label{ortho}
P_{M}^{\perp}(f)=\sum_{I\in {\cal D}_{M}^{c}}\langle f,\psi_{I}\rangle \psi_{I} 
\end{equation}
is to be interpreted in its Schauder basis sense,
$$
\lim_{M'\rightarrow \infty }\| P_{M}^{\perp}(f)
-\sum_{I\in {\cal D}_{M'}\backslash {\cal D}_{M}}\langle f,\psi_{I}\rangle \psi_{I}  \|_{E}=0.
$$
Note that according to Theorem \ref{charofcompact}, an operator $T: E \to E$, is compact if and only if 
$$
\lim_{M\rightarrow \infty }\| P_{M}^{\perp}\circ T\|=0,
$$
where $\| \cdot \|$ is the operator norm.

\subsection{The space $\CMO(\mathbb R)$}\label{subsectionCMO}
We provide now the definition and main properties of the space to which the function $T(1)$ must belong if $T$ is compact.

\begin{definition}
We define $\CMO(\mathbb R)$ as the closure in $\BMO(\mathbb R)$ of the space of continuous functions vanishing at infinity. 
\end{definition}

We note that $\CMO(\mathbb R)$ equipped with the norm $\| \cdot \|_{\BMO}$ is a Banach space. 
The next lemma gives two characterizations of $\CMO(\mathbb R)$: the first in terms of the average deviation from the mean, and the second in terms of a 
wavelet decomposition. See \cite{Roch} for the first, and \cite{LTW} for the second characterization.

\begin{lemma} \label{lem:cmochar}
i) $f\in \CMO(\mathbb R)$ if and only if 
$f\in \BMO(\mathbb R)$ and
\begin{equation}\label{CMO}
\lim_{M\rightarrow \infty } \sup_{I\notin {\mathcal I}_{M}} 
\frac{1}{|I|}\int_{I}\Big|f(x)-\frac{1}{|I|}\int_{I}f(y)dy\Big|dx =0\\
\end{equation}

ii)
$f\in \CMO(\mathbb R)$ if and only if 
$f\in \BMO(\mathbb R)$ and
\begin{equation}\label{CMO2}
\lim_{M\rightarrow \infty} 
\sup_{Q\subset \mathbb R} \Big(\frac{1}{|Q|}
\sum_{\tiny \begin{array}{c}I\notin {\cal D}_{M}\\ I\subset Q \end{array}}|\langle f,\psi_{I}\rangle |^{2}\Big)^{1/2}=0,
\end{equation}
\indent where the supremum is calculated over all intervals $Q \subset \mathbb R$.
\end{lemma}

As a consequence of previous Lemma, $(\psi_I)_{I\in \mathcal{D}}$ is a Schauder basis for $\CMO(\R)$.
We will mainly be using the latter formulation.

\begin{remark}
Considering the comment after Definition \ref{Imdef}, we see that the preceding lemma is also true if we, in line with \cite{Roch}, replace $\mathcal{I}_M$ by $\mathcal{I}'_M$ consisting of those intervals $I$ such that $2^{-M} \leq |I| \leq 2^M$ and $|c(I)| \leq M/2$, and $\mathcal{D}_M$ by $\mathcal{D}'_M = \mathcal{I}'_M \cap \mathcal{D}$.
\end{remark}

We note that the remarks about equality (\ref{ortho}) work well for the spaces $L^{p}(\mathbb R)$, $H^{1}(\mathbb R)$, and also $\CMO(\mathbb R)$, but not for $\BMO(\mathbb R)$. The latter space is not separable and so, it does not contain an unconditional basis. However, the characterization of the norm in $\BMO(\mathbb R)$ by a wavelet basis 
implies that for every $f\in \BMO(\mathbb R)$ we have the equality 
$$
f=\sum_{I\in {\cal D}}\langle f,\psi_{I}\rangle \psi_{I} 
$$
with convergence in the weak* topology $\sigma (\BMO(\mathbb R), H^{1}(\mathbb R))$. This, in turn, implies the quality 
\begin{equation}\label{ortho2}
P_{M}^{\perp}(f)=\sum_{I\in {\cal D}_{M}^{c}}\langle f,\psi_{I}\rangle \psi_{I} 
\end{equation}
with convergence in the same topology, which is interpreted as 
$$
\lim_{M'\rightarrow \infty }\Big|\langle P_{M}^{\perp}(f),g\rangle 
-\sum_{I\in {\cal D}_{M'}\backslash {\cal D}_{M}}\langle f,\psi_{I}\rangle \langle \psi_{I},g\rangle\Big|=0
$$
for all $g\in H^{1}(\mathbb R)$. 
See \cite{LeMe} and \cite{AT} for proofs and more details.

\subsection{Main results}
We now give the statement of the main result in the paper and also the results
in \cite{V} which we will need. 

\begin{theorem}\label{Mainresult2}
Let
$T$ be a linear operator associated with a standard Calder\'on-Zygmund kernel. 

Then, $T$ extends to a compact operator from $L^{\infty}(\mathbb R)$ into 
$\CMO(\mathbb R)$
if and only if $T$ 
is associated with a compact Calder\'on-Zygmund kernel, $T$ satisfies
the weak compactness condition
and 
$T(1), T^{*}(1) \in \CMO(\mathbb R)$. 

Moreover, with the extra assumption $T(1)=T^{*}(1)=0$, $T$ is compact 
from $\BMO(\mathbb R)$ into $\CMO(\mathbb R)$.
\end{theorem}

The analog result appearing in \cite{V} is the following Theorem.

\begin{theorem}\label{Mainresult}
Let
$T$ be a linear operator associated with a standard Calder\'on-Zygmund kernel. 

Then, $T$ extends to a compact operator on $L^p(\mathbb R)$ for $1<p<\infty$ if and only if $T$ 
is associated with a compact Calder\'on-Zygmund kernel, $T$ satisfies
the weak compactness condition
and $T(1), T^{*}(1) \in \CMO(\mathbb R)$. 
\end{theorem}

We now state the key ingredient in the proof of Theorem \ref{Mainresult} and 
also Theorem \ref{Mainresult2}: the so-called bump Lemma, which describes the action of the operator over functions adapted to two different intervals. 

Given two intervals $I$ and $J$, we will denote $K_{min}=J$ and $K_{max}=I$ if $|J|\leq |I|$, while $K_{min}=I$ and $K_{max}=J$ otherwise. 

\begin{proposition}\label{symmetricspecialcancellation}

Let $K$ be a compact Calder\'on-Zygmund kernel with parameter $\delta$. 
Let $N$ sufficiently large depending on $\delta $ and $0<\theta<1$, $0<\delta' <\delta $ depending on $N$.

Let $T:{\mathcal S}_{N}\rightarrow {\mathcal S}_{N}'$ be a linear operator associated with $K$ satisfying the weak compactness condition with parameter $N$ 
and the special cancellation condition $T(1)=0$ and $T^{*}(1)=0$.

Then, 
there exists $C_{\delta'}>0$  such that
for every $\epsilon >0$, all intervals $I, J$ 
and 
all mean zero bump functions $\psi_{I}$, $\psi_{J}$, $L^{2}$-adapted to $I$ and $J$ respectively
with order $N$ and constant $C>0$, 
we have
\begin{equation*}\label{twobump2}
|\langle T(\psi_{I}),\psi_{J}\rangle |\leq  C_{\delta'} C \, \ec(I,J)^{\frac{1}{2}+\delta'}
\rdist(I,J)^{-(1+\delta')}
\Big(F(I_{1},\ldots ,I_{6};M_{T,\epsilon })+\epsilon \Big)
\end{equation*}
where  
$I_{1}=I$, $I_{2}=J$, $I_{3}=\langle I, J\rangle $, $I_{4}=\lambda_{1}\tilde{K}_{max}$, 
$I_{5}=\lambda_{2}\tilde{K}_{max}$,
$I_{6}=\lambda_{2}K_{min}$
with $\lambda_{1}=|K_{max}|^{-1}{\rm diam}(I\cup J)$, 
$\lambda_{2}=(|K_{min}|^{-1}{\rm diam}(I\cup J))^{\theta }$ and 
$\tilde{K}_{max}$ is the translate of $K_{max}$ with the same centre as $K_{min}$.
\end{proposition}

\section{Endpoint estimates}\label{L2}
In this section, we extend the study of compactness for singular integral operators 
to the endpoint case. Namely, we characterize 
those Calder\'on-Zygmund operators that extend compactly as maps from $L^\infty(\mathbb R)$ to $\CMO(\mathbb R)$. 

\subsection{Necessity of the hypotheses}
The necessity of the hypotheses of Theorem \ref{Mainresult2} is a direct consequence of the following result. 
\begin{proposition}
Let $T$ be a linear operator associated with a standard Calder\'on-Zygmund kernel. If $T$ is compact from 
$L^{\infty }(\mathbb R)$ into $\CMO(\mathbb R)$ then, $T$ is compact on $L^{p}(\mathbb R)$ for 
$1<p<\infty $. 
\end{proposition}
\begin{proof}
Since $T$ is bounded from $L^{\infty }(\mathbb R)$ into $\CMO(\mathbb R)$ and it is associated with a 
standard Calder\'on-Zygmund kernel, by \cite{Journe} page 49, $T$ is bounded on $L^{p}(\mathbb R)$ for all 
$1<p<\infty $. Therefore, by interpolation, $T$ is compact on $L^{p}(\mathbb R)$.
\end{proof}

Whence, since in particular $T$ is compact on $L^{2}(\mathbb R)$, by the results in \cite{V} we have that the hypotheses of Theorem \ref{Mainresult2} are satisfied, that is, $T$ is associated with a compact Calder\'on-Zygmund kernel, satisfies the weak compactness condition and $T(1), T^{*}(1)\in \CMO(\mathbb R)$.

\subsection{Wavelet basis}
We devote the first part of this subsection to describe the way 
to choose a wavelet basis of $L^{p}(\mathbb R)$ and $H^{1}(\mathbb R)$ and how we use this basis to decompose the operators under study. In order to do this, we will use the results contained in the books 
\cite{Chui} and \cite{HerWeiss}.

For every function $\psi $ 
and every dyadic interval $I=2^{-j}[k,k+1]$, $j,k\in \mathbb Z$, we denote
$$
\psi_{I}(x)
={\mathcal T}_{l(I)}\mathcal D_{|I|}^{2}\psi (x)
=2^{j/2}\psi(2^{j}x-k),
$$
where $l(I)=\min\{ x: x\in I\}$. 

\begin{theorem}
Let $\psi \in L^{2}(\mathbb R)$ with $\| \psi\|_{L^{2}(\mathbb R)}=1$. Then, 
$\{ \psi_{I}\}_{I\in {\mathcal D}}$ is an orthonormal wavelet basis of $L^{2}(\mathbb R)$
if and only if 
$$
\sum_{k\in \mathbb Z}|\hat{\psi}(\xi +k)|^{2}=1
\hskip20pt, \hskip25pt
\sum_{k\in \mathbb Z}\hat{\psi}(2^{j}(\xi +k))\overline{\hat{\psi}(\xi +k)}=0
$$
for all $\xi \in \mathbb R$ and all $j\geq 1$.
\end{theorem}

\begin{definition}
For any function $f:\mathbb R\rightarrow \mathbb C$, we say that a bounded function $W:[0,\infty )\rightarrow \mathbb R^{+}$ is a radial decreasing $L^{1}$-majorant of $f$ if $|f(x)|\leq W(|x|)$ and $W$ satisfies the following three conditions:
$W\in L^{1}([0,\infty ))$, 
$W$ is decreasing and 
$W (0)<\infty $.
\end{definition}

\begin{theorem}
Let $\psi \in L^{2}(\mathbb R)$ differentiable and such that $\{ \psi_{I}\}_{I\in {\mathcal D}}$ is an orthonormal basis of $L^{2}(\mathbb R)$. We further assume that  $\psi $ and its derivative $\psi'$ have a common radial decreasing $L^{1}$-majorant $W$ satisfying
$$
\int_{0}^{\infty }xW(x)dx<\infty .
$$
Then, the system $(\psi_{I})_{I\in {\mathcal D}}$ is an unconditional basis for $L^{p}(\mathbb R)$ with $1<p<\infty $ and for $H^{1}(\mathbb R)$. 
\end{theorem}

Now, for our particular purposes, we will take $\psi $ satisfying the hypotheses of previous theorems with the additional condition that 
$\psi \in C^{N}(\mathbb R)$ and it is adapted to $[-1/2,1/2]$ with constant $C>0$ and order $N$. Then, 
we remark the crucial fact that for every interval $I\in {\mathcal D}$, every wavelet function 
$\psi_{I}$ is a bump function adapted to $I$ with the same constant $C>0$ and the same order $N$. Several examples of constructions of systems of wavelets with any required order of differentiability can also be found in \cite{HerWeiss}.

In the described setting, the continuity of $T$ with respect the topology of ${\mathcal S}_{N}(\mathbb R)$, 
allows to write
$$
\langle T(f),g\rangle 
=\sum_{I,J\in {\cal D}}\langle f,\psi_{I}\rangle \langle g, \psi_{J}\rangle \langle T(\psi_I),\psi_J\rangle 
$$
for every $f,g\in {\mathcal S}(\mathbb R)$, 
where the sums run over the whole family of dyadic intervals in $\mathbb R$ and convergence is understood in 
the topology of  ${\mathcal S}_{N}(\mathbb R)$.
Furthermore, since
$$
\langle P_{M}(T(f)),g\rangle 
=\langle T(f),P_{M}g\rangle 
=\sum_{I\in {\cal D}}\sum_{J\in {\cal D}_{M}}\langle f,\psi_{I}\rangle 
\langle g, \psi_{J}\rangle \langle T(\psi_I), \psi_J \rangle ,
$$
we have that  
\begin{equation}\label{orthopro0}
\langle P_{M}^{\perp}(T(f)),g\rangle 
=\sum_{I\in {\cal D}}\sum_{J\in {\cal D}_{M}^{c}}\langle f,\psi_{I}\rangle \langle g, \psi_{J} \rangle 
\langle T(\psi_I),\psi_J\rangle 
\end{equation}
where the summation is performed as in equation \eqref{ortho}.

\subsection{Boundedness of a Martingale transform}\label{martingale}
We now study a new Martingale transform. 
Its definition and the proof of its boundedness on $L^{p}(\mathbb R)$ for $1<p<\infty $
appear in the preprint \cite{PoVi}. 
We include here the endpoint result. 

\begin{definition}
Let $(\psi_{I})_{I\in {\mathcal D}}$ be a wavelet basis of $L^{2}(\mathbb R)$. 
Given $k\in \mathbb Z$ and $n\in \mathbb N$, $n\geq 1$, let
$T_{k,n}$ be the operator defined by
$$
T_{k,n}(f)(x)
=\sum_{I\in {\mathcal D}}\sum_{J\in I_{k,n}}\langle f,\psi_{J}\rangle \psi_{I}(x),
$$
 where for each fixed dyadic interval $I$, $I_{k,n}$ is
the family of all dyadic intervals $J$ such that $|I|=2^k|J|$ and $n\leq \rdist(I,J)<n+1$.
\end{definition}

We remind that for every dyadic interval $I$ and each $n\in \mathbb N$ there are $2^{\max(e,0)+1}$ dyadic intervals $J$ such that
$|I|=2^{e}|J|$ and  $n\leq \rdist(I, J)<n+1$. This implies that the cardinality of
$I_{e,n}$ is comparable to $2^{\max(e,0)}$.

In the proposition below, we prove boundedness of this modified Martingale operator.

\begin{proposition}\label{modifiedmartingaletransform} Let $k\in \mathbb Z$ and $n\in \mathbb N$, $n\geq 1$. Then, 
$T_{k,n}$ is bounded on $\BMO(\mathbb R)$. Moreover, 
$$
\| T_{k,n}f\|_{\BMO(\mathbb R)}
\lesssim 2^{\frac{|k|}{2}}(\log(n+1)+\max(-k,0)+1)^{\frac{1}{2}}\| f\|_{\BMO(\mathbb R)}
$$
with implicit constant independent of $f$, $k$ and $n$.
\end{proposition}
\begin{remark} 
By duality and the fact that
$
T_{k,n}^{*}=T_{-k,n}
$,
we have that $T_{k,n}$ is also bounded on $H^1(\mathbb R)$ with
$$
\| T_{k,n}f\|_{H^1(\mathbb R)}
\lesssim 2^{\frac{|k|}{2}}(\log(n+1)+\max(k,0)+1)^{\frac{1}{2}}\| f\|_{H^1(\mathbb R)}
$$
\end{remark}
\proof
Since for any given $f\in {\mathcal S}(\mathbb R)$,
\begin{equation*}
\| T_{k,n}f\|_{\BMO(\mathbb R)}
=\sup_{Q \subset \mathbb R}
\Big(|Q|^{-1}\sum_{I\subset Q}
\Big|\sum_{J\in I_{k,n}}\langle f,\psi_{J}\rangle \Big|^2\, \Big)^{\frac{1}{2}},
\end{equation*}
where the supremum is calculated over all intervals $Q \subset \mathbb R$, 
we will show that 
\begin{align}\label{normofT}
\sum_{I\subset Q}&
\Big|\sum_{J\in I_{k,n}}\langle f,\psi_{J}\rangle \Big|^2
\\
\nonumber
&
\lesssim  2^{|k|}(\log(n+1)+\max(-k,0)+1)\| f\|_{\BMO(\mathbb R)}^{2}|Q|.
\end{align}

In order to compute the double sum, we use an argument that distinguishes between large and small scales
($k\geq 0$ and $k\leq 0$), with a slightly different argument in each case.

We first assume $k\geq 0$. In this case, the cardinality of $I_{k,n}$ is comparable to $2^{k}$ and so,
every interval $I\in {\mathcal D}$ is associated with $2^{k}$ different
intervals $J\in I_{k,n}$. Therefore, by Cauchy's inequality, the contribution of those intervals collected 
in the sum in (\ref{normofT}) can be bounded by
\begin{equation}\label{normofT2}
\sum_{I\subset Q}2^{k}\sum_{J\in I_{k,n}}|\langle f,\psi_{J}\rangle |^2
=2^{k}\sum_{j\in \mathbb N}\sum_{I\in Q^{j}}\sum_{J\in I_{k,n}}|\langle f,\psi_{J}\rangle |^2 ,
\end{equation}
where $Q^{j}=\{ I\in {\mathcal D}: I\subset Q , 2^{-(j+1)}|Q|< |I|\leq 2^{-j}|Q|\}$.

Now,
we separate again into two different cases: when $J\subset 3Q$ and when $J\nsubseteq 3Q $.

1) In the first case, we start by showing that the intervals $J$ in the inner sum of (\ref{normofT2})
only appear at most four times. This will be clear once we prove that given $I\in {\mathcal D}$ and $J\in I_{k,n}$ there exist at most four different intervals $I'\in {\mathcal D}$, $I'\neq I$, such that $J\in I'_{k,n}$. 

If $J\in I_{k,n}\cap I'_{k,n}$ then $|I|=2^{k}|J|=|I'|$. Now, we denote $I_{n}=(I+n|I|)\cup (I-n|I|)$. Since $k\geq 0$, we have that 
$J\subset I_{n}\cap I'_{n}\neq \emptyset $. Then, if $n>1$, this implies $\rdist (I,I')=2n$ and so, $I'=I+n|I|$ or 
$I'=I-n|I|$. On the other hand, if $n=1$, this implies $\rdist (I,I')\in \{1,2\}$ and so, $I'=I+|I|$, $I'=I-|I|$, $I'=I+2|I|$ or $I'=I-2|I|$.

Therefore, 
the terms
in the inner sum of (\ref{normofT2}) corresponding to this case 
can be bounded by a constant times
$$
2^{k}\sum_{J\subset 3Q}|\langle f,\psi_{J}\rangle |^2
\lesssim 2^{k}\| f\|_{\BMO(\mathbb R)}^2|Q|,
$$
which is compatible with the stated bound. 

2) In the second case, for those intervals $I$, $J$ such that $I\subset Q$ and $J\nsubseteq 3Q$, we have
$\diam(I\cup J)>|Q|$. 
Then, for every $I\in Q^{j}$ we get 
$$
n+1>\rdist(I,J)= \frac{\diam (I\cup J)}{|I|}>\frac{|Q|}{|I|}\geq 2^{j},
$$ 
where we have used that $|J|\leq |I|$. 

We now show that, for every $j$, 
the union of the disjoint intervals $J\in I_{k,n}$ when varying $I\in Q^{j}$  has measure at most $2|Q|$. 
For fixed $I$, the union of the disjoint intervals $J\in I_{k,n}$ measures $2|I|$. Moreover, the union of the disjoint intervals 
$I\in Q^{j}$ measures at most $|Q|$. 
Therefore, 
$$
|\bigcup_{I\in Q^{j}}\bigcup_{J\in I_{k,n}}J|\leq \sum_{I\in Q^{j}}\sum_{J\in I_{k,n}}|J|
\leq 2\sum_{I\in Q^{j}}|I|\leq 2|Q| .
$$

This way, 
the relevant contribution of this case to the sum in (\ref{normofT2}) can be bounded by
\begin{align*}
2^{k}\sum_{j=0}^{\log(n+1)}&
\sum_{\tiny \begin{array}{c}I\in Q^{j}\\ J\in I_{k,n}\end{array}}|\langle f,\psi_{J}\rangle |^2
\lesssim 2^{k}\| f\|_{\BMO(\mathbb R)}^2\sum_{j=0}^{\log(n+1)}
\Big| \hspace{-.3cm} 
\bigcup_{\tiny \begin{array}{c}I\in Q^{j}\\ J\in I_{k,n}\end{array}} J\Big| 
\\
&\lesssim 2^{k}\| f\|_{\BMO(\mathbb R)}^2\sum_{j=0}^{\log(n+1)}|Q|
=2^{k}(1+\log(n+1))\| f\|_{\BMO(\mathbb R)}^2|Q|,
\end{align*}
which is the desired bound when $k\geq 0$.

\vskip10pt
For $k\leq 0$ we reason as follows. The cardinality of $I_{k,n}$ is now 
essentially one and there are at most $2^{j}$ intervals $I\in Q^{j}$. But now, 
up to $2^{-k}$ different intervals $I$ of fixed size in the sum (\ref{normofT}) are associated with the same
interval $J$ and so, with the same coefficient $\langle f,\psi_J\rangle$. Then, 
if we denote 
$Q_{k,n}^{j}=\{ J\in I_{k,n}: I\in Q^{j}\}$, we have that 
the terms in the sum 
(\ref{normofT}) corresponding to this case can be bounded by
\begin{equation}\label{normofT3}
\sum_{j\in \mathbb N}
\sum_{J\in Q_{k,n}^{j}}
\min (2^{j},2^{-k})|\langle f,\psi_{J}\rangle |^2,
\end{equation}
where 
now the intervals $J\in Q_{k,n}^{j}$ appearing in the sum are pairwise different. Moreover, 
since $|I|=2^{k}|J|$ and $2^{-(j+1)}|Q|<|I|\leq 2^{-j}|Q|$ 
we get $2^{j+k}\leq |Q|/|J|<2^{j+k+1}$.  

We separate the study into the same two cases as before: $J\subset 3Q$ and $J\nsubseteq 3Q$. 

1) When $I \subset Q$ and $J \subset 3Q$ we have
$$
n\leq \rdist(I,J)=\frac{\diam (I\cup J)}{|J|}\leq \frac{3|Q|}{|J|}<2^{j+k+4}
$$ 
Therefore, 
$j>\log n-k-4\geq -k-4$
and so, the contribution of the intervals in this case to the sum (\ref{normofT3}) can be bounded by
\begin{align*}
\sum_{-k-4<j\leq -k}&
\sum_{\tiny \begin{array}{c} J\in Q_{k,n}^{j}\\
J\subset 3Q\end{array}}2^{j}|\langle f,\psi_{J}\rangle |^2
+\sum_{-k\leq j}
\sum_{\tiny \begin{array}{c} J\in Q_{k,n}^{j}\\J\subset 3Q\end{array}}
2^{-k}|\langle f,\psi_{J}\rangle |^2
\\
&\lesssim 2^{-k}\sum_{J\subset 3Q}|\langle f,\psi_{J}\rangle |^2
\lesssim 2^{-k}\| f\|_{\BMO(\mathbb R)}^2|Q| .
\end{align*}

2) On the other hand, for those $J$ such that $J\nsubseteq 3Q$ we have that 
$$
n+1>\rdist(I,J)=\frac{\diam (I\cup J)}{|J|}>\frac{|Q|}{|J|}\geq 2^{j+k}
$$ 
and so 
$j<\log(n+1)-k$. 
Then, the contribution to sum (\ref{normofT3}) can be estimated by
\begin{align}\label{unionJ}
\nonumber
\sum_{j=0}^{\log(n+1)+|k|}&\min (2^{j},2^{-k})\sum_{J\in Q_{k,n}^{j}}|\langle f,\psi_{J}\rangle |^2
\\
&\leq \sum_{j=0}^{\log(n+1)+|k|}\min (2^{j},2^{-k})
\| f\|_{\BMO(\mathbb R)}^2 \, \, \Big| \hspace{-.3cm} 
\bigcup_{J\in Q_{k,n}^{j}} J\Big| .
\end{align}

We now calculate the measure of the union of those intervals $J\in Q_{k,n}^{j}$. 
If $j\geq -k$ then, from the different $2^{j}$ possible intervals $I\in Q^{j}$, up to $2^{-k}$ of them are associated with the same interval $J$. Then, the union of those 
intervals $J$ has measure
$
\frac{2^{j}}{2^{-k}}|J|\leq |Q|.
$
On the other hand, when $j<-k$, there is only a single interval $J$ associated with all intervals $I\in Q^{j}$, which measures $|J|\leq |Q|2^{-(k+j)}$. Then, 
the union has measure at most $\max(2^{-k-j},1)|Q|$ and thus, we bound \eqref{unionJ} by
\begin{align*}
&\| f\|_{\BMO(\mathbb R)}^2 \sum_{j=0}^{\log(n+1)+|k|}2^{j}\min (1,2^{-k-j})\max(2^{-k-j},1)|Q|
\\
&=\| f\|_{\BMO(\mathbb R)}^2\sum_{j=0}^{\log(n+1)+|k|}2^{-k}|Q|
\\
&=2^{-k}(\log(n+1)+|k|+1)\| f\|_{\BMO(\mathbb R)}^2|Q| .
\end{align*}
This finishes the proof.

\subsection{Sufficiency of the hypotheses: proof of endpoint compactness}
In this subsection, we prove compactness of singular integral operators $T$ as maps from $L^\infty(\mathbb R)$ 
to $\CMO(\mathbb R)$.

To prove this result, we follow the scheme of the original proof of the $T(1)$ Theorem. Namely, we first assume that the special cancellation property $T(1)=T^{*}(1)=0$ holds, and then we tackle the general case with the use of paraproducts. 
Actually, we prove that under the special cancellation conditions the operator $T$ extends compactly from 
$\BMO(\mathbb R)$ into $\CMO(\mathbb R)$.

\subsubsection{The special case: $T(1)=T^{*}(1)=0$}
We start by proving the main result 
under the special cancellation conditions.
\begin{theorem}
\label{BMObounds}
Let
$T$ be a linear operator associated with a compact Calder\'on-Zygmund kernel 
satisfying the weak compactness condition 
and the special cancellation conditions $T(1)=0$ and $T^{*}(1)=0$.

Then, $T$ can be extended to a compact operator from $\BMO(\mathbb R)$ into $\CMO(\mathbb R)$.
\end{theorem}
\proof
Let $(\psi_{I})_{I\in {\mathcal D}}$ be a wavelet basis of $L^{2}(\mathbb R)$ and $H^{1}(\mathbb R)$
with $L^2$-normalized elements. Let $P_{M}$ be the lagom projection operator defined by this basis. 
By the remarks at the end of Subsection \ref{subsectionCMO},
we have that all functions in 
$\BMO(\mathbb R)$ and $H^{1}(\mathbb R)$ can be approximated by functions in 
${\mathcal S}_{N}(\mathbb R)$ with convergence in the weak* topology 
$\sigma (\BMO(\mathbb R), H^{1}(\mathbb R))$ and in the $H^{1}(\mathbb R)$-norm 
respectively. 

Then, by Theorem \ref{charofcompact} with $E=\CMO(\mathbb R)$ equipped with the norm of $\BMO(\mathbb R)$, in order to show compactness of $T$, we 
need to check that
$P_{M}^{\perp}(T_{b})$ converges to zero in the operator norm 
$\| \cdot \|_{\BMO(\mathbb R)\rightarrow \BMO(\mathbb R)}$ when $M$ tends to infinity.
For this, it is enough to prove that 
$
\langle P_{M}^{\perp}(T(f)),g\rangle 
$
tends to zero uniformly for all $f,g \in {\mathcal S}_{N}(\mathbb R)$ in the unit ball of 
$\BMO(\mathbb R)$ and $H^{1}(\mathbb R)$ respectively.

For $f,g\in {\mathcal S}_{N}(\mathbb R)$, we recall 
\begin{equation}\label{orthoprojendpoint}
\langle P_{M}^{\perp}(T(f)),g\rangle 
=\sum_{I\in {\mathcal D}}\sum_{J\in {\cal D}_{M}^{c}}\langle f,\psi_{I}\rangle \langle g,\psi_{J}\rangle 
\langle T(\psi_I),\psi_J\rangle .
\end{equation}
Since this inequality is understood as a limit, 
we can assume that the sums run over finite but arbitrary in size families of dyadic intervals and we will work
to obtain bounds that are are independent of the cardinality of the families.

We start by proving that for every 
 $\epsilon >0$ there is $M_{0}\in \mathbb N$ 
such that for any $M>M_{0}$, 
we have $F(I_{1},\ldots ,I_{6};M_{T,\epsilon })\lesssim \epsilon $ for $I_{i}\in {\cal D}_{M}^{c}$. This will follow as a consequence
of the inequality $F(I;M_{T,\epsilon })\lesssim \epsilon $ for every $I\in {\cal D}_{M}^{c}$. We note that the implicit constants only depend on the admissible functions. 

We first remind that by the definition of the weak compactness condition, 
for $\epsilon >0$  there is $M_{T,\epsilon }>0$ a constant such that 
for any interval $I$ and any $\phi_{I},\varphi_{I}$ bump functions adapted to $I$ with constant $C>0$ and order $N$, we have
$$
|\langle T(\phi_{I}),\varphi_{I}\rangle |\leq C(F_{W}(I;M_{T,\epsilon })+\epsilon) .
$$

We now show that there is  
$M_{0}'\in \mathbb N$ 
depending on $\epsilon $ 
such that for any $M>M_{0}'$, we have 
$F(I_{i};M_{T,\epsilon })=F(I_{1},\ldots ,I_{6};M_{T,\epsilon})\lesssim \epsilon $ when all $I_{i}\in {\cal D}_{M}^{c}$.

By the limit properties of the admissible functions in Definition \ref{admissible}, we have that for fixed $M_{T,\epsilon }>0$,  there is
$M_{0}'\in \mathbb N$, depending on $\epsilon, M_{T,\epsilon }$, with $M_{0}'>M_{T,\epsilon }$,  
such that for any $M>M_{0}'$, we get 
$$
L_{K}(2^{M})+S_{K}(2^{-M})+ D_{K}(M)<\epsilon
$$
and 
$$
L_{W}(2^{M-M_{T,\epsilon }})+S_{W}(2^{-(M-M_{T,\epsilon })})+D_{W}(M/M_{T,\epsilon })<\epsilon .
$$ 

Let $I\in {\cal D}_{M}^{c}$. The claim is proven by considering the following cases: 
\begin{enumerate}
\item If $|I|>2^{M}$ then, since $L_{K}$ and $L_{W}$ are non-increasing, we have  
\begin{align*}
F(I;M_{T,\epsilon})&\lesssim L_{K}(|I|)+L_{W}(|I|/2^{M_{T,\epsilon }})
\\
&\leq L_{K}(2^{M})+L_{W}(2^{M-M_{T,\epsilon }})\lesssim \epsilon .
\end{align*}

\item If $|I|<2^{-M}$ then, since $S_{K}$ and $S_{W}$ are non-decreasing, we have 
\begin{align*}
F(I;M_{T,\epsilon})&\lesssim S_{K}(|I|)+S_{W}(2^{M_{T,\epsilon }}|I|)
\\
&\leq S_{K}(2^{-M})+S_{W}(2^{-(M-M_{T,\epsilon })})
\lesssim \epsilon .
\end{align*}

\item If $2^{-M}\leq |I|\leq 2^{M}$ with $\rdist (I,\mathbb B_{2^{M}})>M$ then, as we saw in the remark after 
Definition \ref{Imdef}, 
$|c(I)|>(M-1)2^{M}$. Therefore,
\begin{align*}
\rdist (I,\mathbb B_{2^{M_{T,\epsilon }}})&\geq 1+\frac{|c(I)|}{\max(|I|,2^{M_{T,\epsilon }})}
\\
&\geq 1+\frac{(M-1)2^{M}}{\max(2^{M},2^{M_{T,\epsilon }})}
\geq M .
\end{align*}
We can apply a similar reasoning to show that we also have $\rdist (I,\mathbb B)>M$. 
Then, since $D_{W}$ is non-increasing, we have 
\begin{align*}
F(I;M_{T,\epsilon})&\lesssim  D_{K}(\rdist (I,\mathbb B))
+D_{W}(M_{T,\epsilon }^{-1}\rdist (I,\mathbb B_{2^{M_{T,\epsilon }}}))
\\
&\leq D_{K}(M)+D_{W}(M/M_{T,\epsilon })\lesssim \epsilon  .
\end{align*}

\end{enumerate}

Therefore, there is finally  
$M_{0}\in \mathbb N$ 
depending on $\epsilon $ 
such that for any $M>M_{0}$, we have 
\begin{enumerate}
\item $F(I_{i};M_{T,\epsilon })=F(I_{1},\ldots ,I_{6};M_{T,\epsilon})\lesssim \epsilon $ when all $I_{i}\in {\cal D}_{M}^{c}$ 
\item $M^{-\frac{\delta}{2}}+M^{\frac{3}{2}}2^{-M\frac{\delta }{2}}+\sum_{e\geq M}2^{-e\delta }e^{1/2}<\epsilon .$
\end{enumerate}

Now, for every $\epsilon >0$ and chosen $M_{0}\in \mathbb N$, we are going to prove that for all 
$M>M_{0}$ we have
$$
|\langle P_{2M}^{\perp}(T(f)),g\rangle |\lesssim \epsilon ,
$$
with the implicit constant depending on $\delta >0$ and the constant given by the wavelet basis.

We first parametrize the terms in (\ref{orthoprojendpoint}) according to eccentricity and relative distance 
to obtain
\begin{equation}\label{compactendpoint}
\langle P_{2M}^{\perp}(T(f)),g\rangle 
=\sum_{e\in \mathbb Z}\sum_{n\in \mathbb N}
\sum_{\tiny \begin{array}{c}J{\in \cal D}_{2M}^{c}\end{array}}
\sum_{I\in J_{e,n}}
\langle f,\psi_{I}\rangle \langle g,\psi_{J}\rangle \langle T(\psi_I),\psi_J\rangle ,
\end{equation}
where for fixed eccentricity $e\in \mathbb Z$, relative distance $n\in \mathbb N$ and every given interval $J$,
$$
J_{e,n}=\{ I:|I|=2^{e}|J|, n\leq \rdist(I, J)< n+1 \} .
$$

By Proposition \ref{symmetricspecialcancellation} we have 
$$
|\langle T(\psi_I),\psi_J\rangle |
\lesssim 2^{-|e|(\frac{1}{2}+\delta )}n^{-(1+\delta )}(F(I_{i};M_{T,\epsilon })+\epsilon ) ,
$$
where $I_{1}=I$, $I_{2}=J$, $I_{3}=\langle I,J\rangle$ , $I_{4}=\lambda_{1}\tilde{K}_{max}$, 
$I_{5}=\lambda_{2}\tilde{K}_{max}$ and  $I_{6}=\lambda_{2}K_{min}$.
with parameters $\lambda_{1},\lambda_{2}\geq 1$ explicitly 
stated in the mentioned Proposition.  
To simplify notation, we will simply write $F(I_{i})$. 
We also note that the implicit constant might depend on $\delta $ and the wavelet basis, but it is universal otherwise. Therefore, 
\begin{align}\label{moduloin2}
|\langle P_{2M}^{\perp}(T(f)),g\rangle | & \lesssim 
\sum_{e\in \mathbb Z}\sum_{n\in \mathbb N}2^{-|e|(\frac{1}{2}+\delta )}n^{-(1+\delta )}
\\
\nonumber
&\sum_{\tiny \begin{array}{c}J{\in \cal D}_{2M}^{c}\end{array}}\sum_{I\in J_{e,n}}
\big( F(I_{i})+\epsilon \big)
|\langle f,\psi_{I}\rangle | |\langle g,\psi_{J}\rangle | .
\end{align}

Now, in order to estimate \eqref{moduloin2}, we divide the study into six cases:
\vskip5pt
\noindent
\hspace{-.5cm}
\begin{minipage}{7cm}
\begin{enumerate}
\item $I_{i}\notin {\cal D}_{M}$ for all $i=1,\ldots ,6$
\item $I\in {\cal D}_{M}$
\item $\langle I\cup J\rangle \in {\cal D}_{M}$
\end{enumerate}
\end{minipage}
\hspace{-.5cm}
\begin{minipage}{7cm}
\begin{enumerate}
\item[(4)] $I\notin {\cal D}_{M}$ but $\lambda_{1}\tilde{K}_{max}\in {\cal D}_{M}$
\item[(5)] $I\notin {\cal D}_{M}$ but $\lambda_{2}\tilde{K}_{max}\in {\cal D}_{M}$
\item[(6)] $I\notin {\cal D}_{M}$ but $\lambda_{2}K_{min} \in {\cal D}_{M}$
\end{enumerate}
\end{minipage}
\vskip10pt

1) In the first case we have
$F(I_{i})<\epsilon $, and 
thus, we can bound the contribution of the corresponding terms to (\ref{compactendpoint}) by
\begin{equation}\label{compactonCMO}
\epsilon \sum_{e\in \mathbb Z}\sum_{n\in \mathbb N}2^{-|e|(\frac{1}{2}+\delta  )}n^{-(1+\delta )}
\sum_{\tiny \begin{array}{c}J{\in \cal D}_{2M}^{c}\end{array}}\sum_{I\in J_{e,n}}
|\langle f,\psi_{I}\rangle | |\langle g,\psi_{J}\rangle | .
\end{equation}

Since as said, we consider that $I$ and $J$ run over finite families of intervals, 
we can define $\tilde{f}=\sum_{I}|\langle f,\psi_{I}\rangle |\psi_{I}$ and 
$\tilde{g}=\sum_{J}|\langle g,\psi_{J}\rangle |\psi_{J}$, so that $\langle \tilde{f},\psi_I\rangle = |\langle f,\psi_{I}\rangle |$ 
and similarly for $\tilde{g}$. 
Then, for any interval $Q \subset \mathbb R$ we have
$$
\| {\mathcal P}_{Q}(\tilde{f})\|_{L^2(\mathbb R)}^2
\leq \sum_{I\subset Q}|\langle f, \psi_{I}\rangle |^2
=\| {\mathcal P}_{Q}(f)\|_{L^2(\mathbb R)}^2 ,
$$
where ${\mathcal P}_{Q}(f)=\sum_{I\subset Q}\langle f, \psi_{I}\rangle \psi_{I}$ is the classical projection operator defined in \eqref{regularproj}. Therefore 
$\| \tilde{f}\|_{\BMO(\mathbb R)}\leq \| f\|_{\BMO(\mathbb R)}$ and, by a duality argument, we also have 
$\| \tilde{g}\|_{H^1(\mathbb R)}\leq \| g\|_{H^1(\mathbb R)}$.

With this, we get for the inner sums in \eqref{compactonCMO}:
\begin{align*}
\sum_{\tiny \begin{array}{c}J{\in \cal D}_{2M}^{c}\end{array}}&\sum_{I\in J_{e,n}}
|\langle f,\psi_{I}\rangle ||\langle g,\psi_{J}\rangle |
\leq \sum_{\tiny \begin{array}{c}J\in {\mathcal D}\end{array}}\sum_{I\in J_{e,n}}
|\langle f,\psi_{I}\rangle ||\langle g,\psi_{J}\rangle |
\\
&=\sum_{\tiny \begin{array}{c}J\in {\mathcal D}\end{array}}\sum_{I\in J_{e,n}}
\langle \tilde{f},\psi_{I}\rangle \langle \tilde{g},\psi_{J}\rangle
=\Big\langle \tilde{g}, \sum_{J\in {\mathcal D}} \sum_{\tiny \begin{array}{c}I\! \in \! J_{e,n}\end{array}}
\langle \tilde{f},\psi_{I}\rangle \psi_{J}\Big\rangle
\\
&=\big\langle \tilde{g}, T_{e,n}(\tilde{f})\big\rangle
\leq \| \tilde{g}\|_{H^{1}(\mathbb R)}\| T_{e,n}(\tilde{f})\|_{\BMO(\mathbb R)}
\\
&\leq 2^{\frac{|e|}{2}}(\log(n+1)+|e|+1)^{\frac{1}{2}}\| \tilde{f}\|_{\BMO(\mathbb R)}\| \tilde{g}\|_{H^1(\mathbb R)}, 
\end{align*}
where the last inequality is due to Proposition \ref{modifiedmartingaletransform}. 
Notice also that $\log(n+1)\leq 2\log n\leq 2\delta^{-1}n^{\delta }$. 

This way,  (\ref{compactonCMO}) can be bounded by a constant times
\begin{align*}
\epsilon \sum_{e\in \mathbb Z}\sum_{n\in \mathbb N}
&2^{-|e|(\frac{1}{2}+\delta )}n^{-(1+\delta )}
2^{\frac{|e|}{2}}(n^{\delta }+|e|+1)^{\frac{1}{2}}\| f\|_{\BMO(\mathbb R)}\| g\|_{H^1(\mathbb R)}
\\
&\lesssim \epsilon \sum_{e\in \mathbb Z}2^{-|e|\delta }|e|^{\frac{1}{2}}\sum_{n\geq 1}
n^{-(1+\delta -\frac{\delta}{2})}\| f\|_{\BMO(\mathbb R)}\| g\|_{H^1(\mathbb R)}
\\
&\lesssim \epsilon \| f\|_{\BMO(\mathbb R)}\| g\|_{H^1(\mathbb R)}.
\end{align*}

In the remaining cases, we will not use the smallness of $F$. Instead, we will use 
the particular geometrical disposition of the intervals $I$ and $J$, which make either their eccentricity or their relative distance very extreme. 
We recall that the intervals  $I$ and $J$ in the sum (\ref{moduloin2}) satisfy $|I|=2^{e}|J|$ and 
$n\leq \rdist(I,J)<n+1$. 

2) We deal first with the case when 
$I\in {\cal D}_{M}$, that is, when  
$2^{-M}\leq |I|\leq 2^{M}$ and $\rdist(I,\mathbb B_{2^{M}})\leq M$.
Notice that, since $F$ is bounded, we can estimate $F(I_{i})+\epsilon \lesssim 1$. 

Since $J\in {\mathcal D}_{2M}^{c}$, we separate the study into three cases: $|J|>2^{2M}$, $|J|<2^{-2M}$
and $2^{-2M}\leq |J|\leq 2^{2M}$ with $\rdist(J,\mathbb B_{2^{2M}})>2M$.

2.1) In the case $|J|>2^{2M}$, since 
$2^{e}|J|=|I|\leq 2^{M}$, we have $2^{e}\leq 2^{M}|J|^{-1}\leq 2^{-M}$, that is, 
$e\leq -M$. 
Therefore, the calculations developed in case 1) allow to bound the corresponding terms in (\ref{moduloin2}) by 
\begin{align*}
&\sum_{\tiny \begin{array}{c}e\leq -M\end{array}}
\sum_{n\geq 1}
2^{-|e|(\frac{1}{2}+\delta )}n^{-(1+\delta )}
\sum_{J\in {\cal D}_{2M}^{c}
} \sum_{I\in J_{e,n}}|\langle f,\psi_{I}\rangle ||\langle g,\psi_{J}\rangle |
\\
&\leq \hspace{-.3cm} \sum_{\tiny \begin{array}{c}e\leq -M\end{array}}
\sum_{n\geq 1}
2^{-|e|(\frac{1}{2}+\delta )}n^{-(1+\delta )}
 2^{\frac{|e|}{2}}(n^{\delta}+|e|+1)^{\frac{1}{2}}\| f\|_{\BMO(\mathbb R)}\| g\|_{H^1(\mathbb R)}
\\
&\lesssim \Big(\sum_{\tiny \begin{array}{c}e\leq -M\end{array}}
2^{-|e|\delta }|e|^{\frac{1}{2}}\sum_{n\geq 1}
n^{-(1+\frac{\delta}{2})}\Big)\| f\|_{\BMO(\mathbb R)}\| g\|_{H^1(\mathbb R)}
\\
&
\lesssim \epsilon \| f\|_{\BMO(\mathbb R)}\| g\|_{H^1(\mathbb R)}
\end{align*}
by the choice of $M$. This finishes this case.

2.2) The case $|J|<2^{-2M}$ is symmetrical and amounts to changing $e\leq -M$ by $e\geq M$
in the previous case. 

2.3) In the case when $2^{-2M}\leq |J|\leq 2^{2M}$  and $\rdist(J,\mathbb B_{2^{2M}})\geq 2M$, we have that  
$|J|=2^{k}$ with $-2M\leq k\leq 2M$ and $|c(J)|\geq (2M-1)2^{2M}$.  
Since $I\in {\mathcal D_{M}}$, we also have 
\begin{align*}
M&\geq \rdist(I,\mathbb B_{2^{M}})=2^{-M}\diam (I\cup \mathbb B_{2^{M}})
\\
&\geq 2^{-M}(2^{M-1}+|I|/2+|c(I)|)
\geq 2^{-M}(2^{M-1}+|c(I)|)
\end{align*}
and then, $|c(I)|\leq (M-1/2)2^{M}$. This implies 
\begin{align*}
|c(I)-c(J)|&\geq |c(J)|-|c(I)|
\\
&\geq (2M-1)2^{2M}-(M-1/2)2^{M}\geq M2^{2M}.
\end{align*} 
This way, since $\max(|I|,|J|)\leq 2^{2M}$, we get
$$
n+1> \rdist(I,J)=\frac{\diam(I\cup J)}{\max(|I|,|J|)}\geq \frac{|c(I)-c(J)|}{\max(|I|,|J|)}
$$
$$
\geq 2^{-2M}M2^{2M}=M .
$$

Therefore, as in previous case, we bound the relevant terms in (\ref{moduloin2}) by a constant times
\begin{align*}
\sum_{\tiny \begin{array}{c}e\in \mathbb Z\end{array}}
&\sum_{n\geq M-1}
2^{-|e|(\frac{1}{2}+\delta )}n^{-(1+\delta )}
\sum_{J\in {\cal D}_{2M}^{c}} \sum_{I\in J_{e,n}}|\langle f,\psi_{I}\rangle ||\langle g,\psi_{J}\rangle |
\\
&\lesssim \Big(\sum_{\tiny \begin{array}{c}e\in \mathbb Z\end{array}}
2^{-|e|\delta }|e|^{\frac{1}{2}}\sum_{n\geq M-1}
n^{-(1+\frac{\delta }{2})}\Big)\| f\|_{\BMO(\mathbb R)}\| g\|_{H^1(\mathbb R)}
\\
&\lesssim M^{-\frac{\delta }{2}} \| f\|_{\BMO(\mathbb R)}\| g\|_{H^1(\mathbb R)}
<\epsilon \| f\|_{\BMO(\mathbb R)}\| g\|_{H^1(\mathbb R)}
\end{align*}
again by the choice of $M$.

3) Now, we deal with the case when $\langle I, J\rangle \in {\mathcal D}_{M}$, that is, when 
$2^{-M}\leq |\langle I, J\rangle |\leq 2^{M}$ and $\rdist (\langle I,J\rangle ,\mathbb B_{2^{M}})\leq M$. Both 
inequalities imply that $2^{-M}\leq \diam(I\cup J)\leq 2^{M}$ and $|c(\langle I, J\rangle )|\leq M2^{M}$.

Moreover, we have that $c(\langle I, J\rangle )=1/2\big(c(I)+c(J)+\alpha (|I|-|J|)\big)$ with $\alpha \in [-1,1]$. Then, 
\begin{align}\label{contra}
|c(I)+c(J)|&\leq 2|c(\langle I, J\rangle )|+||I|-|J||
\\
\nonumber
&\leq 2M2^{M}+|\langle I, J\rangle |
\leq (2M+1)2^{M}
\end{align}

3.1) When $|J|> 2^{2M}$ we have 
that $|\langle I, J\rangle|\geq |J|>2^{2M}$ implies $\langle I, J\rangle \notin {\cal D}_{M}$ and so, we do not
need to consider this case. 

3.2) When $2^{-2M}\leq |J|\leq 2^{2M}$ with $\rdist(J,\mathbb B_{2^{2M}})\geq 2M$, we have that
$|c(J)|>(2M-1)2^{2M}>M2^{M}$ . 

If $\sign \, c(I)=-\sign \, c(J)$ we have 
\begin{align*}
|\langle I, J\rangle|&=\diam(I\cup J)\geq |c(I)-c(J)|
\\
&=|c(I)|+|c(J)|>|c(J)|
>M2^{M},
\end{align*}
which is contradictory with $\langle I, J\rangle \in {\mathcal D}_{M}$. 

Otherwise, if $\sign \, c(I)=\sign \, c(J)$ we have 
$$
|c(I)+c(J)|= |c(I)|+|c(J)|>M2^{M}, 
$$
which is now contradictory with \eqref{contra}. 

So, we do not need to consider this case either.

3.3) The remaining case is when $|J|< 2^{-2M}$.
If $e\geq 0$ then,
$$
n+1> |I|^{-1}\diam(I\cup J)=2^{-e}|J|^{-1}|\langle I, J\rangle |\geq 2^{-e}2^{2M}2^{-M}=2^{M-e} .
$$

Meanwhile, if $e\leq 0$ we have
$$
n+1> \rdist(I,J)=|J|^{-1}\diam(I\cup J)\geq 2^{2M}2^{-M}=2^{M} .
$$

Therefore, we bound the relevant part of  (\ref{moduloin2}) by a constant times
\begin{align*}
&\sum_{\tiny \begin{array}{c}e\geq 0\end{array}}
\sum_{n\geq \max(2^{M-e}-1,1)}
2^{-|e|(\frac{1}{2}+\delta )}n^{-(1+\delta )}
\sum_{J\in {\cal D}_{2M}^{c}} \sum_{I\in J_{e,n}}|\langle f,\psi_{I}\rangle ||\langle g,\psi_{J}\rangle |
\\
&\hskip40pt +\sum_{\tiny \begin{array}{c}e\leq 0\end{array}}
\sum_{n\geq 2^{M}-1}
2^{-|e|(\frac{1}{2}+\delta )}n^{-(1+\delta )}
\sum_{J\in {\cal D}_{2M}^{c}} \sum_{I\in J_{e,n}}|\langle f,\psi_{I}\rangle ||\langle g,\psi_{J}\rangle |
\\
&\leq \Big(\hspace{-.1cm}\sum_{\tiny \begin{array}{c}0\leq e\leq M-1\end{array}}\hspace{-.3cm}
2^{-|e|\delta }|e|^{\frac{1}{2}}\hspace{-.3cm}\sum_{n\geq 2^{M-e}-1}\hspace{-.3cm}
n^{-(1+\frac{\delta }{2})}
+\sum_{\tiny \begin{array}{c}M\leq e\end{array}}\hspace{-.3cm}
2^{-|e|\delta }|e|^{\frac{1}{2}}\sum_{n\geq 1}
n^{-(1+\frac{\delta }{2})}
\\
&\hskip40pt + \sum_{\tiny \begin{array}{c}e\leq 0\end{array}}
2^{-|e|\delta }|e|^{\frac{1}{2}}\sum_{n\geq 2^{M-1}}n^{-(1+\frac{\delta }{2})}
\Big)\| f\|_{\BMO(\mathbb R)}\| g\|_{H^1(\mathbb R)}
\\
&\lesssim \Big( \hspace{-.5cm}\sum_{\tiny \begin{array}{c}0\leq e\leq M-1\end{array}}\hspace{-.5cm}
2^{-e\delta }|e|^{\frac{1}{2}}2^{-(M-e)\frac{\delta }{2}}
+\hspace{-.1cm}\sum_{\tiny \begin{array}{c}M\leq e\end{array}}\hspace{-.2cm}
2^{-e\delta }|e|^{\frac{1}{2}}+2^{-M\frac{\delta }{2}}\Big)\| f\|_{\BMO(\mathbb R)}\| g\|_{H^1(\mathbb R)}
\\
&\lesssim \Big( 2^{-M\frac{\delta }{2}}M^{\frac{3}{2}}+\hspace{-.1cm}\sum_{\tiny \begin{array}{c}M\leq e\end{array}}\hspace{-.2cm}
2^{-e\delta }|e|^{\frac{1}{2}}+2^{-M\frac{\delta }{2}}\Big)\| f\|_{\BMO(\mathbb R)}\| g\|_{H^1(\mathbb R)}
\\
&\lesssim \epsilon \| f\|_{\BMO(\mathbb R)}\| g\|_{H^1(\mathbb R)}
\end{align*}
by the choice of $M$.
 
6) We deal now with the case $\lambda_{2}K_{min}\in {\cal D}_{M}$, that is, 
$2^{-M}\leq |\lambda_{2}K_{min}|\leq 2^{M}$ and $\rdist (\lambda_{2}K_{min} ,\mathbb B_{2^{M}})\leq M$.

6.1) When $|J|> 2^{2M}$, we have two cases. Whenever $e>0$
then, $K_{min}=J$ and so, 
$|\lambda_{2}J|\geq |J|\geq 2^{2M}$ which is contradictory with $\lambda_{2}J\in {\cal I}_{M}$.

On the other hand, when $e\leq 0$ we have $K_{min}=I$ and  $|I|\leq |\lambda_{2}I|\leq 2^{M}$. 
Then, $2^{e}=|I|/|J|\leq 2^{-M}$ and so, $e\leq -M$. Therefore, the arguments of the case 2.1) show that the corresponding part 
of (\ref{moduloin2}) can be bounded by $\epsilon  \| f\|_{\BMO(\mathbb R)}\| g\|_{H^1(\mathbb R)}$.

6.2) When $2^{-2M}\leq |J|\leq 2^{2M}$ with $\rdist(J,\mathbb B_{2^{2M}})\geq 2M$, we have  
$|c(J)|>(2M-1)2^{2M}$. Now, we divide into the same two cases. 

When $e\geq 0$, we know $K_{min}=J$ and so, 
$2^{-M}\leq |\lambda_{2}J|\leq 2^{M}$ with $\rdist(\lambda_{2}J,\mathbb B_{2^{M}})\leq M$. This leads to  
the following contradiction:
$$
M\geq \rdist(\lambda_{2}J,\mathbb B_{2^{M}})
> 
2^{-M}|c(J)|
\geq (2M-1)2^{M}.
$$
On the other hand, when $e\leq0 $ we have $K_{min}=I$ and then, 
$|c(I)|=|c(\lambda_{2}I)|\leq (M-1)2^{M}$. This implies $|c(I)-c(J)|> M2^{2M}$ and
$$
n+1> \rdist(I,J)
\geq \frac{|c(I)-c(J)|}{|J|}
\geq M .
$$
Then, the same arguments developed in the case 2.3) provide the bound $\epsilon \| f\|_{\BMO(\mathbb R)}\| g\|_{H^1(\mathbb R)}$.

6.3) When $|J|< 2^{-2M}$, we proceed as follows. 
If $e\geq 0$, we have $K_{\min}=J$ and so, 
$|\lambda_{2}J|\geq 2^{-M}$. This implies $\lambda_{2}\geq 2^{-M}|J|^{-1}> 2^{M}$ and
$$
2^{M}<\lambda_{2}=\Big(\frac{\diam(I\cup J)}{|J|}\Big)^{\theta}
= \Big(\frac{|I|}{|J|}\Big)^{\theta }\rdist(I,J)^{\theta}
<2^{e\theta }(n+1)^{\theta}
$$
Meanwhile, if $e\leq 0$, we have $K_{\min}=I$ and then, $|\lambda_{2}I|\geq 2^{-M}$. We also have 
$|I|\leq |J|\leq 2^{-2M}$. 
All this implies $\lambda_{2}\geq 2^{-M}|I|^{-1}> 2^{M}$ and
$$
2^{M}<\lambda_{2}
= \Big(\frac{\diam(I\cup J)}{|I|}\Big)^{\theta}
=\Big(\frac{|J|}{|I|}\Big)^{\theta }\rdist(I,J)^{\theta}
<2^{-e\theta }(n+1)^{\theta }
$$

Then, since $\theta <1$, we get $n+1>2^{-|e|}2^{\frac{M}{\theta }}>2^{-|e|}2^{M}$ and so,   
previous arguments show that  the relevant part 
of (\ref{moduloin2}) can be bounded by
\begin{align*}
\Big(\sum_{\tiny \begin{array}{c}e\in \mathbb Z\end{array}}
&2^{-|e|\delta }|e|^{\frac{1}{2}}\sum_{n\geq 2^{-|e|}2^{M}-1}
n^{-(1+\frac{\delta }{2})}\Big)\| f\|_{\BMO(\mathbb R)}\| g\|_{H^1(\mathbb R)}
\\
&\lesssim \Big(\sum_{\tiny \begin{array}{c}e\in \mathbb Z\end{array}}
2^{-|e|\delta }|e|^{\frac{1}{2}}2^{|e|\frac{\delta }{2}}2^{-M\frac{\delta }{2}}
\Big)\| f\|_{\BMO(\mathbb R)}\| g\|_{H^1(\mathbb R)}
\\
&
\lesssim 2^{-M\frac{\delta }{2}}\| f\|_{\BMO(\mathbb R)}\| g\|_{H^1(\mathbb R)}
\leq \epsilon \| f\|_{\BMO(\mathbb R)}\| g\|_{H^1(\mathbb R)}.
\end{align*}

Finally, we note that similar type of calculations are enough to deal with the two remaining cases 4) and 5). This completely finishes the proof of Theorem \ref{BMObounds}.

\subsection{The general case}
For the proof of compactness 
in
the general case, that is, without the special cancellation conditions,  
we follow the same scheme as in the proof of the classical $T(1)$ theorem. 
When $b_1=T(1)$ and $b_2=T^{*}(1)$ are arbitrary functions in $\CMO(\mathbb R)$, we construct compact paraproducts $T_{b}$ associated with compact 
Calder\'on-Zymund kernels such that
$T_{b_1}(1)=b_1$, $T_{b_1}^{*}(1)=0$.
Then, the operator 
$$
\tilde{T}=T-T_{b_1}-T_{b_2}^{*}
$$ 
satisfies the hypotheses of Theorem \ref{BMObounds} and so, $\tilde{T}$ is compact from 
$\BMO(\mathbb R)$ to $\CMO(\mathbb R)$. Finally, since the operators  $T_{b_1}$ and $T_{b_2}^{*}$ are compact 
from $L^{\infty }(\mathbb R)$ to $\CMO(\mathbb R)$
by construction, we deduce that  
the initial operator $T$ is also compact from $L^{\infty }(\mathbb R)$ to $\CMO(\mathbb R)$.

We remark that, as we will later see in full detail, the appropriate paraproducts are exactly the same ones as in the classical setting, with the only difference that the parameter functions $b_{i}$ belong to the space $\CMO(\mathbb R)$ instead of 
$\BMO(\mathbb R)$.

As in Proposition \ref{paraproducts1},  
we use a wavelet basis $(\psi_I)_{I\in {\mathcal D}}$ of $L^2(\mathbb R)$ and $H^{1}(\mathbb R)$ such that each $\psi_{I}$ is an 
$L^{2}$-normalized bump function supported and adapted to $I$ with constant $C$ and order $N$.

We now denote by $\phi$ a positive bump function supported and adapted to $[-1/2,1/2]$ with order $N$ and integral one. Then, we have that 
$0\leq \phi (x)\leq C(1+|x|)^{-N}$ and $|\phi'(x)|\leq C(1+|x|)^{-N}$. 
Let $(\phi_{I})_{I\in {\mathcal D}}$  be
the family of bump functions defined by $\phi_I={\mathcal T}_{c(I)}{\mathcal D}_{|I|}^{1}\phi$. Therefore, 
each 
$\phi_{I}$ is an $L^{1}$-normalized bump function adapted to $I$, that is, it satisfies
 $\phi_{I}(x)\leq C|I|^{-1}(1+|I|^{-1}|x-c(I)|)^{-N}$  and $|\phi_{I}'(x)|\leq C|I|^{-2}(1+|I|^{-1}|x-c(I)|)^{-N}$.

\begin{proposition}\label{paraproducts1}
Given $b \in \CMO(\R)$, we define the operator
$$
T_b(f)=\sum_{I\in {\mathcal D}} \langle b, \psi_{I}\rangle \langle f, \phi_{I} \rangle \psi_{I},
$$
where $\psi_I$ and $\phi_I$ are as described above.

Then, 
$T_b$ and $T_b^{*}$ are associated with a compact Calder\'on-Zygmund kernel, 
and they are both compact from $L^{\infty }(\mathbb R)$ to $\CMO(\mathbb R)$. 
Furthermore, 
$
\langle T_b(1),g\rangle =\langle b,g\rangle
$
and 
$
\langle T_b(f),1\rangle =0
$, 
for all $f,g\in {\mathcal S}(\mathbb R)$.
\end{proposition}

\proof
In \cite{V} we showed that $T_b$ and $T_{b}^{*}$ belong to the class of operators for which the theory applies, that is, the integral representation of Definition \ref{intrep} holds with
operator kernel satisfying the Definition \ref{prodCZ} of a compact Calder\'on-Zygmund kernel.

For the proof of compactness of $T_{b}$, it is sufficient to verify that
$\langle P_{M}^{\perp}(T_{b})(f), g\rangle $
tends to zero for all $f \in L^\infty(\R)$
and $g \in {\mathcal S}(\mathbb R)$ uniformly in the unit ball of $L^\infty (\mathbb R)$ and $H^1(\mathbb R)$ respectively.  
Since $g\in H^{1}(\mathbb R)$, 
we have
$
P_{M}^{\perp }(g)=\sum_{I\in {\mathcal D}_{M}^{c}}\langle g, \psi_{I}\rangle \psi_{I}
$.

We note that, by the classical $T(1)$ theory, we already now that the operator is bounded from 
$L^{\infty }(\mathbb R)$ to $\CMO(\mathbb R)$ and so, the expression $T_{b}(f)$ is completely meaningful.

Moreover, 
since $(\psi_{I})_{I\in \cal D}$ can be chosen so that it is also a wavelet basis on $\CMO(\mathbb R)$ (see the comment in Lemma \ref{lem:cmochar}), we have 
$P_{M}^{\perp }(b)\in \BMO(\mathbb R)$ and
$
P_{M}^{\perp }(b)=\sum_{I\in {\mathcal D}_{M}^{c}}\langle b, \psi_{I}\rangle \psi_{I}
$. 
With this,
\begin{align*}
\langle P_{M}^{\perp}(T_b(f)),g\rangle &=\langle T_b(f),P_{M}^{\perp}(g)\rangle 
=\sum_{I\in {\mathcal D}} \langle b, \psi_{I}\rangle \langle f, \phi_{I}\rangle 
\langle P_{M}^{\perp}(g),\psi_{I}\rangle
\\
&=\sum_{I\in {\mathcal D}_{M}^{c}} \langle b, \psi_{I}\rangle \langle f, \phi_{I}\rangle \langle g,\psi_{I}\rangle
=\sum_{I\in {\mathcal D}} \langle P_{M}^{\perp}(b), \psi_{I}\rangle \langle f, \phi_{I}\rangle \langle g,\psi_{I}\rangle
\end{align*}
that is, 
\begin{equation}\label{projofparaproduct}
\langle P_{M}^{\perp}(T_b(f)),g\rangle =\langle T_{P_{M}^{\perp}(b)}(f),g\rangle .
\end{equation}

Then, boundedness of $T_{P_{M}^{\perp}(b)}$ from $L^\infty (\mathbb R)$ to $\BMO(\mathbb R)$ implies 
$$
|\langle P_{M}^{\perp}(T_{b})(f), g\rangle |
\lesssim \| P_{M}^{\perp}(b)\|_{\BMO(\mathbb R)}\|f\|_{L^{\infty }(\mathbb R)}\|g\|_{H^{1}(\mathbb R)}.
$$
Since $\lim_{M\rightarrow \infty }\| P_{M}^{\perp}(b)\|_{\BMO(\mathbb R)}=0$, 
the inequality above finally proves that $T_b(f)$
is compact from
$L^\infty(\mathbb R)$ into $\CMO(\mathbb R)$.

The proof that 
$
{T_b}^{*}
$
is compact from $L^\infty(\mathbb R)$ to
$\CMO(\mathbb R)$ is slightly different since ${T_b}^{*}$ does not satisfy the analogue to 
\eqref{projofparaproduct}. Thus, we prove instead the dual compactness for $T_{b}$. 
By \eqref{projofparaproduct} and boundedness of 
$T_{b}$ from $H^{1}(\mathbb R)$ to $L^{1}(\mathbb R)$, we have
$$
|\langle P_{M}^{\perp}(T_{b})(f), g\rangle |
= |\langle (T_{P_{M}^{\perp}(b)}(f), g\rangle |
\lesssim \| P_{M}^{\perp}(b)\|_{\BMO(\mathbb R)}\|f\|_{H^{1}(\R)}\|g\|_{L^{\infty}(\mathbb R)}.
$$
This proves that $T_{b}$ is compact from $H^{1}(\mathbb R)$ to $L^{1}(\mathbb R)$ and so, by duality 
$T_{b}^{*}$ is compact from $L^{\infty }(\mathbb R)$ to $\BMO(\mathbb R)$. But this obviously implies that 
$\lim_{M\rightarrow \infty }\| P_{M}^{\perp}(T_{b}^{*})(f)\|_{\BMO(\mathbb R)}=0$ 
uniformly in the unit ball of $H^{1}(\mathbb R)$ and thus, the range of $T_{b}^{*}$ is actually in 
$\CMO(\mathbb R)$.

\section{Compactness of a perturbation of the Cauchy transform}
In this section we apply our main theorem to demonstrate the
compactness of a certain perturbation of the Cauchy transform, for
Lipschitz paths in the complex plane satisfying a $\CMO$-condition.
The example illustrates with special clarity the scope and 
methodology of the new theory since the computations involved 
are essentially variations of the well known calculations pertaining 
to the study of the Cauchy transform in the classical $T(1)$-theory.
We note that a $T(b)$-theorem for
compactness in several dimensions is already under development, and it
could be of further use in the compactness theory of Cauchy-type
operators.

We start by giving the following definition.
\begin{definition}
We denote by $L^\infty_\CMO(\R)$ the closed subspace $L^\infty(\R) \cap \CMO(\R)$ of $L^\infty(\mathbb R)$.  
\end{definition}

Let $A: \R \to \R$ be an absolutely continuous function such that $A' \in L^\infty_\CMO(\R)$, and let $\Gamma \subset \C$ be the curve given by the parametrization $z(t) = t + iA(t)$, $t \in \R$. Given points $z(x), z(t) \in \Gamma$, we denote by $\sigma_{z(x),z(t)}, \tau_{z(x),z(t)} \in \Gamma$ the points 
\begin{equation*}
\sigma_{z(x),z(t)} = z\left( x-\frac{1}{4}(x-t)\right ), \quad \tau_{z(x),z(t)} = z\left( x-\frac{3}{4}(x-t)\right),
\end{equation*}
lying in between $z(x)$ and $z(t)$ with respect to the parametrization of $\Gamma$. 

The application we present concerns a perturbation of the Cauchy transform associated with $\Gamma$. Namely, define $T_\Gamma : L^p(\Gamma) \to L^p(\Gamma)$ by
\begin{equation*}
T_\Gamma f(z) = 2\int_\Gamma \frac{f(w)}{z-w + 2\overline{(\sigma_{z,w}-\tau_{z,w})}} \, ds(w), \quad f \in L^p(\Gamma), \, z \in \Gamma,
\end{equation*}
where $ds$ denotes the arc length measure on $\Gamma$. Note that if $z = z(x)$ and $w = z(t)$, then
\begin{equation*}
\Re (z-w + 2\overline{(\sigma_{z,w}-\tau_{z,w})}) = 2 \Re(z-w),
\end{equation*}
and
\begin{multline} \label{Imparteq}
\Im (z-w + 2\overline{(\sigma_{z,w}-\tau_{z,w})}) = \\
A(x)-A(t)-2A(x-\frac{1}{4}(x-t)) + 2A(x-\frac{3}{4}(x-t))
\end{multline}
In analogy with the Hilbert transform, we also introduce the operator $H_\Gamma : L^p(\Gamma) \to L^p(\Gamma)$,
\begin{equation*}
H_\Gamma f(z) = \int_\Gamma \frac{f(w)}{\Re(z-w)} \, ds(w), \quad f \in L^p(\Gamma), \, z \in \Gamma.
\end{equation*}
One might surmise that there is a sufficient amount of cancellation in \eqref{Imparteq} to cause $T_\Gamma - H_\Gamma$ to be compact on $L^p(\Gamma)$. We will apply the results of this paper to prove exactly this when $\|A'\|_\infty$ is sufficiently small. 
\begin{proposition} \label{applprop}
Suppose that $A: \R \to \R$ is absolutely continuous, and that $A' \in L^\infty_\CMO(\R)$. Then, there exists an $\eta > 0$ such that $T_\Gamma - H_\Gamma$ is compact on $L^p(\Gamma)$, $1 < p < \infty$, whenever $\|A'\|_\infty < \eta$.
\end{proposition}
Moving over to the real line, we have formally that
\begin{equation} \label{gammatoReq}
(T_\Gamma f- H_\Gamma f)(z(t)) =  \sum_{n = 1}^\infty \left(\frac{-i}{2}\right)^n T_n\left( f \cdot \sqrt{1+|A'|^2} \right)(t),
\end{equation}
 where $T_n: L^p(\R) \to L^p(\R)$ is the operator associated to the kernel 
\begin{equation*}
K_n(x,t) = \frac{(A(x)-A(t) - 2A(x-\frac{1}{4}(x-t)) + 2A(x-\frac{3}{4}(x-t)))^n}{(x-t)^{n+1}}.
\end{equation*}
The expression for $K_1$ is reminiscent of a double difference of $A$. Operators associated to such kernels have received attention by Coifman and Meyer \cite{CoifMey78}. Note that $K_n$ is anti-symmetric for each $n$. In what follows, we will prove that each $K_n$ is a compact Calder\'on-Zygmund kernel and that 
$T_{n}$ satisfies 
the weak compactness condition with appropriate bounds. Moreover, through an inductive procedure, we shall compute $T_n(1)$ and check its membership to $\CMO$. This way, we will deduce that $T_n$ is compact and 
obtain the existence of a constant $C > 0$ such that $\| T_n \|_{L^p \to L^p} \leq C^n \|A'\|_\infty^n$. Then, by setting $\eta = 2/C$, we will have finally proven Proposition \ref{applprop}.

Note that we may write
\begin{equation*}
K_n(x,t) = \frac{1}{x-t}\left( \frac{\int_t^xA'(z) \, dz}{x-t} - \frac{\int_{x-\frac{3}{4}(x-t)}^{x-\frac{1}{4}(x-t)} A'(z) \, dz}{\frac{1}{2}(x-t)} \right)^n.
\end{equation*}
The inner expression can be interpeted as the difference of two averages of $A'$. 
From the estimate
$$
|f_I - f_J| \leq \frac{1}{|J|} \int_J |f-f_I| \, dt \leq \frac{2}{|I|} \int_I | f - f_I| \, dt,
$$
where $J \subset I$ are two intervals such that $|I| = 2|J|$ and $f_I$ denotes the average of $f$ on $I$, we deduce
\begin{equation} \label{cmoineq}
| K_n(x,t) | \lesssim 2^n\| A' \|_{\BMO}^n\frac{1}{|x-t|}.
\end{equation}

Demonstrating the smoothness condition of Definition \ref{prodCZ} is more involved. Let $x,t,t' \in \R$ with $ 0 < 2|t-t'| \leq |x-t|$. We denote $G_n(x,t) = (x-t)^{n+1}K_n(x,t)$
for notational convenience and note that
\begin{multline*}
K_n(x,t) - K_n(x,t') =  \\ G_n(x,t') \left (\frac{(x-t')^{n+1}-(x-t)^{n+1}}{(x-t)^{n+1}(x-t')^{n+1}} \right) + \frac{G_n(x,t) - G_n(x,t')}{(x-t)^{n+1}} .
\end{multline*}
Regarding the first term of this decomposition, there exists, by the mean value theorem, a $\lambda$ between $t$ and $t'$, and therefore satisfying $|t-\lambda| < |t-t'|$, such that
\begin{multline*}
G_n(x,t') \left (\frac{(x-t')^{n+1}-(x-t)^{n+1}}{(x-t)^{n+1}(x-t')^{n+1}} \right) = \\ (n+1)\frac{G_n(x,t')}{(x-t')^n}\frac{(x-\lambda)^n}{(x-t)^n} \frac{t-t'}{(x-t)(x-t')}.
\end{multline*}
For $M > 0$, let 
\begin{equation*}
F_{1,n}(M) = \sup_{\substack{I_{x,t} \in \mathcal{I}_M^c \\ 2|t-t'| \leq |x-t|}} \left |\frac{G_n(x,t')}{(x-t')^n} \right|,
\end{equation*} 
where $I_{x,t}$ is the interval with endpoints $x$ and $t$, and $\mathcal{I}_M$ is the set of intervals $I$ with center $c(I)$ such that $2^{-M} < |I| < 2^M$ and $|c(I)| < \frac{M}{2}$. Clearly $F_{1,n}$ is decreasing, and from the assumption that $A' \in \CMO$ in conjunction with the estimate \eqref{cmoineq} it follows that $\|F_{1,n}\|_\infty \lesssim 2^n \|A'\|^n_{\BMO}$ and $\lim_{M \to \infty} F_{1,n}(M) = 0$. This gives us control of the first term, 
\begin{multline*}
\left | G_n(x,t') \left (\frac{(x-t')^{n+1}-(x-t)^{n+1}}{(x-t)^{n+1}(x-t')^{n+1}} \right) \right | \\ \lesssim (n+1) \left( \frac{3}{2} \right)^n F_{1,n}\left( \max \left( \left| \log_2(|x-t|) \right|, |x+t| \right) \right) \frac{|t-t'|}{|x-t|^2}.
\end{multline*}

To deal with the second term we will consider the cases $n=1$ and $n \geq 2$ separately. Suppose first that $n \geq 2$. Applying the mean value theorem, there exists a $\lambda$ with
\begin{equation*}
|A(t) + 2A(x-\frac{1}{4}(x-t)) - 2A(x-\frac{3}{4}(x-t)) - \lambda| < |G_1(x,t) - G_1(x,t')|
\end{equation*}
and such that
\begin{equation*}
\frac{G_n(x,t) - G_n(x,t')}{(x-t)^{n+1}} = n\frac{G_1(x,t) - G_1(x,t')}{t-t'}\frac{(A(x)-\lambda)^{n-1}}{(x-t)^{n-1}} \frac{t-t'}{(x-t)^2}.
\end{equation*}
At this point the condition $A' \in L^\infty$ comes into play, since it is necessary for estimating the first factor;
\begin{equation*}
\left| \frac{G_1(x,t) - G_1(x,t')}{t-t'} \right| \lesssim \| A' \|_\infty.
\end{equation*}
On the other hand, introducing
\begin{equation*}
F_{2,n}(M) = \sup_{\substack{I_{x,t} \in \mathcal{I}_M^c \\ 2|t-t'| \leq |x-t|}} \left | \frac{A(x)-\lambda}{x-t} \right|^{n-1},
\end{equation*}
we have that $A' \in \CMO$ again implies that $\| F_{2,n} \|_\infty \lesssim 4^n \|A'\|_{\BMO}^n$ and $\lim\limits_{M \to \infty} F_{2,n}(M) = 0$. Therefore,
\begin{multline*}
\left |  \frac{G_n(x,t) - G_n(x,t')}{(x-t)^{n+1}} \right | \lesssim \\ n \| A' \|_\infty F_{2,n}\left( \max \left( \left| \log_2 (|x-t|) \right|, |x+t|\right) \right) \frac{|t-t'|}{|x-t|^2}.
\end{multline*}

When $n=1$, the previous argument fails. Instead, we pick a $\delta$, $0 < \delta < 1$, and write
\begin{equation*}
\left| \frac{G_1(x,t) - G_1(x,t')}{(x-t)^2} \right| = \left| \frac{G_1(x,t) - G_1(x,t')}{t-t'} \right | \frac{|t-t'|^{1-\delta}}{|x-t|^{1-\delta}} \frac{|t-t'|^\delta}{|x-t|^{1+\delta}}.
\end{equation*}
Define 
\begin{equation*}
F_{2,1}(M) = \sup_{\substack{I_{x,t} \in \mathcal{I}_M^c \\ 2|t-t'| \leq |x-t|}} \left| \frac{G_1(x,t) - G_1(x,t')}{t-t'} \right | \frac{|t-t'|^{1-\delta}}{|x-t|^{1-\delta}}.
\end{equation*}
It is clear that $\| F_{2,1}\|_\infty \lesssim \|A'\|_\infty$. We prove now that $\lim_{M \to \infty} F_{2,1}(M) = 0$. For suppose that $\varlimsup_{M \to \infty} F_{2,1}(M) = \ell > 0$. Then there exists a sequence $(M_k)$ with $M_k \to \infty$ and corresponding sequences $(x_k), (t_k),$ and $(t'_k)$ such that $I_{x_k,t_k} \in \mathcal{I}_{M_k}^c$, $2|t_k-t'_k| \leq |x_k-t_k|$, and 
\begin{equation} \label{contrineq}
\left| \frac{G_1(x_k,t_k) - G_1(x_k,t'_k)}{t_k-t'_k} \right | \frac{|t_k-t'_k|^{1-\delta}}{|x_k-t_k|^{1-\delta}} > \ell/2.
\end{equation}
There could not exist a constant $C > 0$ such that $$\frac{|t_k-t'_k|}{|x_k-t_k|} \geq C$$ for all $k$, for then 
\begin{multline*} 
\left| \frac{G_1(x_k,t_k) - G_1(x_k,t'_k)}{t_k-t'_k} \right | \frac{|t_k-t'_k|^{1-\delta}}{|x_k-t_k|^{1-\delta}} \leq \\ \frac{2^{1-\delta}}{C} \left| \frac{G_1(x_k,t_k) - G_1(x_k,t'_k)}{x_k-t_k} \right | \to 0, \, \, k \to \infty, 
\end{multline*}
by the fact that $A' \in \CMO$. Hence it must be that $$ \varliminf_{k\to\infty} \frac{|t_k-t'_k|}{|x_k-t_k|} = 0.$$ This also contradicts \eqref{contrineq}, however, since the first factor of the left hand side is bounded, seeing as $A' \in L^\infty$. By this contradiction we conclude that $\lim_{M \to \infty} F_{1,2}(M) = 0$.

Appealing to these estimates and the anti-symmetry of $K_n$  we may easily construct a set of admissible functions $L_n$, $S_n$ and $D_n$ so that the conditions of Definition \ref{prodCZ} are fulfilled with $\delta = 1$ for $n \geq 2$ and every $\delta < 1$ for $n=1$. Hence, $K_n$ is a compact Calder\'on-Zygmund kernel.

We turn now to the verification of the weak compactness condition. For every compact interval $I$ with center $c(I)$, we introduce the kernel
\begin{equation*}
K_n^I(x,t) = |I|K_n(|I|x+c(I), |I|t+c(I)),
\end{equation*}
and note that 
\begin{equation} \label{KIeq}
|K_n^I(x,t)| \leq \frac{1}{|x-t|} F_{3,n} \left( \max \left( \left| \log_2 (|I||x-t|) \right|, ||I|(x+t) + 2c(I)| \right) \right),
\end{equation}
where 
\begin{equation*}
F_{3,n}(M) = \sup_{I_{x,t} \in \mathcal{I}_M^c} \left |\frac{G_n(x,t)}{(x-t)^n} \right|.
\end{equation*}
As before it is clear that $F_{3,n}$ is a decreasing function with $\|F_{3,n}\|_\infty \lesssim 2^n \|A'\|_{\BMO}^n$ and $\lim_{M \to \infty} F_{3,n}(M) = 0$. For $\phi \in \mathcal{S}(\R)$, we write $\phi_I(x) = |I|^{-1/2}\phi\left( \frac{x-c(I)}{|I|} \right)$. Given $\phi, \varphi \in \mathcal{S}(\R)$, we have by the anti-symmetry of $K_n$ that
\begin{align*}
2 \langle T_n \varphi_I, \phi_I \rangle &= \int_{\R^2} K_n(x,t) (\phi_I(x) \varphi_I(t) - \phi_I(t) \varphi_I(x) ) \, dx \, dt \\
&= \int_{\R^2} K_n^I(x,t) (\phi(x) \varphi(t) - \phi(t) \varphi(x) ) \, dx \, dt.
\end{align*}
Since $|K_n^I(x,t)| \lesssim 2^n |x-t|^{-1}$ uniformly in $I$, there exists for each pair ($\varphi$,$\phi$) a constant $C_{\varphi,\phi}$ depending only on a finite number of Schwarz class seminorms, with the following property: for every $\varepsilon > 0$ there is an $M \geq 1$, independent of $\varphi$ and $\phi$, such that 
\begin{equation*}
\left| \, \int\limits_{\{|x-t| > M\} \cup \{|x-t| < \frac{1}{M} \} \cup \{|x+t| > M \} } 
\hspace{-1.5cm} K_n^I(x,t) (\phi(x) \varphi(t) - \phi(t) \varphi(x) ) \, dx \, dt \right | \leq 2^nC_{\varphi,\phi} \varepsilon.
\end{equation*}
To see this, simply note that
\begin{equation*}
\frac{\phi(x)\varphi(t)-\phi(t)\varphi(x)}{x-t} \in \mathcal{S}(\R^2)
\end{equation*}
is a Schwarz function of two variables. Furthermore, in view of \eqref{KIeq} we have
\begin{multline} \label{WCCineq}
\left| \, \int\limits_{\{\frac{1}{M} < |x-t| < M\} \cap \{|x+t|<M\}} K_n^I(x,t) (\phi(x) \varphi(t) - \phi(t) \varphi(x) ) \, dx \, dt \right | \\ \leq C_n C'_{\varphi,\phi} F_{3,n} \left( \max \left( \log_2 (\frac{|I|}{M}), -\log_2 (|I|M), \frac{\rdist(I, \mathbb{D}_1)}{M} \right) \right),
\end{multline}
for some constants $C_n$ and $C'_{\varphi,\phi}$, depending only on $n$ and a finite number of seminorms of $\varphi$ and $\phi$, respectively. Note in particular that $|x+t| < M$ and $|c(I)| > M|I|$ imply that
$$||I|(x+t) + 2c(I)| > |c(I)| \gtrsim \frac{\rdist(I, \mathbb{D}_1)-1}{M}.$$
Together with the trivial facts that $|I||x-t| > |I|/M$ and $|I||x-t| < |I|M$ when $(x,t)$ lies in the domain of integration of \eqref{WCCineq}, and that $\frac{\rdist(I, \mathbb{D}_1)}{M} \leq 2$ when $|c(I)| \leq M|I|$, we obtain the desired estimate in \eqref{WCCineq} for an appropriate $C_n$.  We conclude that $T_n$ satisfies the weak compactness condition. For future reference we also record the implied bound on the weak boundedness constant of $T_n$ present in the above considerations. Namely
\begin{equation*}
|\langle T_n \varphi_I, \phi_I \rangle| \lesssim 2^n \|A'\|_{\BMO}^n C'_{\varphi,\phi}.
\end{equation*}

Finally, we shall show that $T_n(1)$ belongs to $\CMO$ by evaluating it inductively in a principal value sense. The justifications for these computations are analogous to those that appear in considerations of Cauchy type operators in connection with the classical $T(1)$ theory, see for example Christ \cite{Christ}. 

For $x$ with $|x| < r < R$ and $\varepsilon > 0$, integrate by parts to obtain that
\begin{multline*}
\int_{\substack{|t| < R \\ |t-x| > \varepsilon}} K_n(x,t) \, dt = \frac{1}{n}\left[\frac{G_n(x,t)}{(x-t)^n}\right]_{t=-R}^{t=R}  - \frac{1}{n} \left [ \frac{G_n(x,t)}{(x-t)^n}\right]_{t=x-\varepsilon}^{t=x+\varepsilon} + \\ \int_{\substack{|t| < R \\ |t-x| > \varepsilon}} \frac{G_{n-1}(x,t)}{(x-t)^n}(A'(t) + \frac{1}{2} A'(x-\frac{1}{4}(x-t)) - \frac{3}{2} A'(x-\frac{3}{4}(x-t))) \, dt,
\end{multline*}
with the understanding that $G_0 \equiv 1$. Splitting the latter integral into three parts according to its summands and making the linear changes of variables $x-t = 4(x - z)$ and $x-t = \frac{4}{3}(x-w)$ in the last two terms we find that
\begin{multline} \label{T1comp}
\int_{\substack{|t| < R \\ |t-x| > \varepsilon}} K_n(x,t) \, dt = \frac{1}{n}\left[\frac{G_n(x,t)}{(x-t)^n}\right]_{t=-R}^{t=R}  - \frac{1}{n} \left [ \frac{G_n(x,t)}{(x-t)^n}\right]_{t=x-\varepsilon}^{t=x+\varepsilon} \\ +  \int_{\substack{|t| < R \\ |t-x| > \varepsilon}} \frac{G_{n-1}(x,t)}{(x-t)^n} A'(t) \, dt +  \frac{2}{4^n}\int_{\substack{|x-4(x-z)| < R \\ |z-x| > \varepsilon/4}} \frac{G_{n-1}(x,x-4(x-z))}{(x-z)^n}A'(z) \, dz 
\\ -  2\left(\frac{3}{4}\right)^n \int_{\substack{|x-\frac{4}{3}(x-w)| < R \\ |w-x| > 3\varepsilon/4}} \frac{G_{n-1}(x,x-\frac{4}{3}(x-w))}{(x-w)^n}  A'(w) \, dw
\end{multline}
Since $A' \in \CMO$ it is clear that the first two terms tend to zero, uniformly for $|x| < r$, as $\varepsilon \to 0$ and $R \to \infty$. 

Suppose now that $n=1$. Seeing as $r$ is arbitrary, we then find in the limit that
\begin{equation*}
T_1(1) = H(A') + \frac{1}{2}H(A') - \frac{3}{2}H(A') = 0.
\end{equation*}
where $H$ denotes the usual Hilbert transform. Note that the Hilbert transform is bounded as a map $H: \CMO \to \CMO$. 

At this point we have verified the compactness of $T_1$ on $L^p$, $1 < p < \infty$, and as a map $T_1 : L^\infty_\CMO \to \CMO$. We now proceed with the inductive step to prove the same for $T_n$, $n \geq 2$. In this case passing to the limit in \eqref{T1comp} gives
\begin{equation}\label{Tn1eq}
T_n(1) = T_{n-1}(A') + \frac{2}{4^n}\widetilde{T}_{n-1}(A') - 2\left(\frac{3}{4}\right)^n \widehat{T}_{n-1}(A'),
\end{equation}
 where $\widetilde{T}_{n-1}$ and $\widehat{T}_{n-1}$ are the operators associated to the kernels 
\begin{align*}
\widetilde{K}_{n-1}(x,z) &= \frac{G_{n-1}(x,x-4(x-z))}{(x-z)^n}, \\ \widehat{K}_{n-1}(x,w) &= \frac{G_{n-1}(x,x-\frac{4}{3}(x-w))}{(x-w)^n}.
\end{align*}
These kernels are very similar in character to $K_{n-1}$, and all computations performed up to this point can be repeated with minor modifications for them. In particular, $\widetilde{K}_{n-1}$ and $\widehat{K}_{n-1}$ are compact Calder\'on-Zygmund kernels, $\widetilde{T}_{n-1}$ and $\widehat{T}_{n-1}$ satisfy the weak compactness condition, $\widetilde{T}_{1}(1) = \widehat{T}_{1}(1) = 0$ and for $n \geq 2$, both $\widetilde{T}_{n}(1)$ and $\widehat{T}_{n}(1)$ are linear combinations of $T_{n-1}(A')$, $\widetilde{T}_{n-1}(A')$ and $\widehat{T}_{n-1}(A')$ with coefficients exponential in $n$. 

Using these results and the fact that $A' \in L^\infty \cap \CMO$, we obtain by induction that $T_n: L^p \to L^p$, $1 < p < \infty$ and $T_n: L^\infty_\CMO \to \CMO$ are compact maps for $n \geq 1$. Furthermore, by inspecting the constants in the above calculations and appealing to classical $T(1)$ theory \cite{Christ}, we obtain bounds on the corresponding operator norms; there exists a constant $C > 0$ such that
\begin{equation*}
\| T_n \|_{L^p \to L^p} \leq C^n \|A'\|_\infty^n, \quad \| T_n \|_{L^\infty_\CMO \to \CMO} \leq C^n \|A'\|_\infty^n.
\end{equation*}
We conclude that $T_\Gamma-H_\Gamma:L^p(\Gamma)\to L^p(\Gamma)$ is compact when $\|A'\|_\infty < 2/C$, hence finishing the proof of Proposition \ref{applprop}. \qed

\bibliographystyle{plain}

\end{document}